\documentclass[11pt]{article}

\usepackage{amsmath}
\usepackage{amsfonts}
\usepackage{color}
\usepackage{float}   
\usepackage{pifont}  
\usepackage{amssymb}
\usepackage{graphicx,bm}
\usepackage{enumerate}
\usepackage{amsthm,amscd}
\usepackage[all]{xy}

 \allowdisplaybreaks

\usepackage[linewidth=1pt,framemethod=TikZ]{mdframed}   

 \usepackage[colorlinks, linkcolor=red, anchorcolor=green, citecolor=blue]{hyperref}

\newtheorem{theorem}{Theorem}[section]
\newtheorem{corollary}[theorem]{Corollary}

\newtheorem{conjecture}[theorem]{Conjecture}

\newtheorem{lemma}[theorem]{Lemma}
\newtheorem{proposition}[theorem]{Proposition}

\newtheorem{problem}[theorem]{Problem}

\begin{document}

\title{A spectral Erd\H{o}s--Faudree--Rousseau theorem}

\author{Yongtao Li$^{1}$\quad 
 Lihua Feng$^{1,}$\footnote{Corresponding authors. 
This paper was published on  Journal of Graph Theory. \\ 
E-mail addresses:  
\url{ytli0921@hnu.edu.cn} (Y. Li), 
\url{fenglh@163.com} (L. Feng), 
\url{ypeng1@hnu.edu.cn} (Y. Peng)}
\quad
 Yuejian Peng$^{2,*}$  \\
{\small $^{1}$School of Mathematics and Statistics,  Central South University, Changsha, China} \\ 
 {\small $^{2}$School of Mathematics, Hunan University, Changsha, China }  
 }

\date{\today}

\maketitle

\vspace{-0.8cm}

\begin{abstract} 
A well-known theorem of Mantel states that every 
$n$-vertex graph with more than $\lfloor n^2/4\rfloor $ edges contains a triangle. An interesting problem in extremal graph theory 
studies the minimum number of edges contained in triangles among
graphs with a prescribed number of vertices and edges. 
 Erd\H{o}s, Faudree and Rousseau (1992) showed that 
a graph on $n$ vertices with more than $\lfloor n^2/4\rfloor $ edges 
contains at least $2\lfloor n/2\rfloor +1$ edges in triangles. 
Such edges are called triangular edges. 
In this paper, we present a spectral version of 
the result of Erd\H{o}s, Faudree and Rousseau. 
Using the supersaturation-stability and the spectral technique, 
we prove that  every $n$-vertex graph 
$G$ with  $\lambda (G) \ge \sqrt{\lfloor n^2/4\rfloor}$ contains at least $2 \lfloor {n}/{2} \rfloor -1$ triangular edges, unless 
$G$ is a balanced complete bipartite graph. 
The method in our paper has some interesting applications. 
Firstly, the supersaturation-stability can be used to revisit  a conjecture of Erd\H{o}s concerning with the booksize of a graph, which was initially proved by Edwards (unpublished), and independently by Khad\v{z}iivanov and Nikiforov (1979). 
Secondly, our method can improve the bound on the order $n$ of the spectral extremal graph when we forbid the friendship graph as a substructure. We drop the condition that requires the order $n$ to be sufficiently large, which was investigated by Cioab\u{a}, Feng, Tait and Zhang (2020) using the triangle removal lemma. 
Thirdly, this method can be utilized to deduce the classical stability 
for odd cycles and it gives more concise bounds on parameters. 
Finally, the supersaturation-stability could be applied to 
deal with the spectral graph problems on counting triangles, 
which was recently studied by Ning and Zhai (2023). 
 \end{abstract}

{{\bf Key words:} Extremal graph theory;  triangular edges; spectral  radius. }

{{\bf 2010 Mathematics Subject Classification.}  05C35; 05C50.}


\section{Introduction} 
Extremal  combinatorics is increasingly becoming 
a fascinating mathematical discipline as well as an essential
component of many mathematical areas, and it has experienced an impressive
growth in recent years. 
Extremal combinatorics concerns the problems  
of determining the maximal or the minimal size of a combinatorial object that satisfies certain properties. 
One of the most important problems is the 
so-called Tur\'{a}n-type problem, which has played an important role in the development 
of extremal combinatorics.  
More precisely,  
the Tur\'{a}n-type questions usually 
study the maximum possible number of edges in
a graph that does not contain a specific subgraph. 
Such kind of questions could be 
viewed as the cornerstone of extremal graph theory and have been studied
extensively in the literature.  

\medskip 
A graph $G$ is {\it $F$-free} if it does not contain a subgraph isomorphic to $F$. 
  For example, every bipartite graph is triangle-free. 
A classical result in extremal graph theory is Mantel's  theorem \cite{Bollobas78}, 
which  asserts that every triangle-free graph on $n$ 
vertices contains at most 
$\lfloor n^2/4\rfloor$ edges. 
This result is tight 
by considering the bipartite Tur\'{a}n graph $T_{n,2}$, where 
$T_{n,2}$ is a complete bipartite graph whose 
two vertex parts have sizes as equal as possible.  
Equivalently, each graph on $n$ vertices
with more than $\lfloor n^2/4\rfloor $ edges must contain a triangle. 

\medskip 
There are several results in the literature that guarantee something stronger than just one triangle. 
 For example, in 1941, 
 Rademacher (unpublished paper, see Erd\H{o}s \cite{Erd1955,Erdos1964}) proved that such graphs contain at least $\lfloor n/2\rfloor$ triangles. 
 After this result, 
Erd\H{o}s \cite{Erd1962a,Erd1962b}
 showed that there exists a small constant
 $c>0$ such that if $n$ is large enough and $1\le q<cn$, then
 every $n$-vertex graph with $\lfloor n^2/4\rfloor +q$
 edges   has at least $q\lfloor {n}/{2}\rfloor$ triangles.
Furthermore, Erd\H{o}s conjectured the constant $c={1}/{2}$, which was finally confirmed by
Lov\'{a}sz and Simonovits \cite{LS1975,LS1983} 
in 1975. They  proved that 
if $1\le q <{n}/{2}$ is a positive integer and $G$ is an $n$-vertex graph with
$e(G)\ge \lfloor {n^2}/{4} \rfloor + q$,  
then $G$ contains at least $q \lfloor {n}/{2}\rfloor $
triangles.  
We refer the readers to \cite{XK2021,LM2022-Erd-Rad,BC2023}
for  recent generalizations on the Erd\H{o}s--Rademacher problem. 
Moreover, Lov\'{a}sz and Simonovits \cite{LS1983} 
also studied the supersaturation problem 
for cliques in the case $q=o(n^2)$. 
For $q=\Omega (n^2)$,  
this problem turns out to be notoriously difficult.
Some recent progress was presented by 
Razborov \cite{Raz2008}, Nikiforov \cite{Niki2011},  Reiher \cite{Rei2016}, Liu, Pikhurko and Staden 
\cite{LPS2020}. In addition, the supersaturation 
problems for color-critical graphs were studied by Mubayi \cite{Mub2010}, and Pikhurko and Yilma \cite{PY2017}.

 \subsection{Minimizing the number of triangular edges}

 In this paper, we shall consider the supersaturation problem 
from a different point of view.  
 An edge is called {\it triangular} if it is contained in a triangle. 
 We shall consider the problem on counting 
 the number of triangular edges, rather than the number of triangles. 
The first result was obtained by 
 Erd\H{o}s, Faudree and Rousseau \cite{EFR1992}, who 
provided a tight bound on 
the number of triangular edges in any 
 $n$-vertex graph 
with more than $\lfloor n^2/4\rfloor$ edges.

\begin{theorem}[Erd\H{o}s--Faudree--Rousseau, 1992] 
\label{thm-EFR}
Let $G$ be a graph with $n$ vertices and 
\[ e(G)> e(T_{n,2}).\] 
Then $G$ has at least $2 \lfloor \frac{n}{2} \rfloor +1$ triangular edges. 
\end{theorem}

 This bound is the best possible simply by adding an edge to the larger vertex part of 
 the balanced  complete bipartite graph.  
Motivated by the problem about the number of triangles, 
it is natural to ask how many triangular edges an $n$-vertex 
graph with $m$ edges must have, where 
$m$ is an integer satisfying $\lfloor n^2/4\rfloor < m \le {n \choose 2}$. 
Indeed, this problem was recently studied by 
F\"{u}redi and Maleki \cite{FM2017} 
as well as Gruslys and Letzter \cite{GL2018}. 
Given integers $a,b,c$, let 
$G(a,b,c)$ denote the graph on 
$n=a+b+c$ vertices, which consists of a clique $A$ 
of size $a$ and two independent sets $B$ and $C$ 
of sizes $b$ and $c$ respectively, such that 
all edges between $B$ and $A\cup C$ induces a complete bipartite graph $K_{b,a+c}$. 
In other words, the graph $G(a,b,c)$ can be obtained from 
$K_{b,a+c}$ by embedding a clique of order $a$ into the part of size $a+c$. 
Note that $G(a,b,c)$ has ${a \choose 2} + (a+c)b$ edges and 
it has ${a \choose 2}+ab = m-bc$ triangular edges. 
In 2017, F\"{u}redi and Maleki \cite{FM2017}  
conjectured that the minimizers of the number of triangular
edges are graphs of the form $G(a,b,c)$ or 
subgraphs of such graphs. 

\begin{conjecture}[F\"{u}redi--Maleki, 2017] 
\label{conj-FM}
Let $m>  n^2/4 $ and $G$ be an $n$-vertex graph with $m$ edges that minimizes the number of triangular edges. 
Then $G$ is isomorphic to a subgraph of $G(a,b,c)$ 
for some $a,b,c$. 
\end{conjecture}

Particularly, F\"{u}redi and Maleki \cite{FM2017}  
proposed a numerical conjecture, which states that 
every $n$-vertex graph with $m$ edges has at least 
$g(n,m)$ triangular edges, where 
\[ g(n,m) = \min
\left\{m -bc: a+b+c =n, {a \choose 2} + b(a+c) \ge m \right\}. \]
We remark that Conjecture \ref{conj-FM} 
characterizes the structures of the minimizers, 
while the latter conjecture gives a lower bound only. 
By using a generalization of Zykov's symmetrization 
method, 
F\"{u}redi and Maleki \cite{FM2017} 
showed a lower bound: 
if $G$ is a graph on $n$ vertices with $m > n^2/4$ edges, 
then $G$ has at least 
$ g(n,m) - 3n/2 $ 
 triangular edges.  
Soon after, Gruslys and Letzter \cite{GL2018} 
proved an exact version of the result of F\"{u}redi and Maleki. 
Let $\mathbf{NT}(G)$ be the set of non-triangular edges of $G$. 
The following result was established in \cite{GL2018}. 

\begin{theorem}[Gruslys--Letzter, 2018] 
\label{thm-GL}
There is $n_0$ such that 
for any graph $G$ on  $n\ge n_0$ vertices, there exists 
a graph $H=G(a,b,c)$ on $n$ vertices such that $e(H)\ge e(G)$  and $|\mathbf{NT}(H)| \ge |\mathbf{NT}(G)|$. 
\end{theorem}

Theorem \ref{thm-GL} shows that  for sufficiently large $n$, 
the minimum number of triangular edges among all 
 $n$-vertex graphs with at least 
$m$ edges is achieved by the graph $G(a,b,c)$ or its subgraph 
for some $a,b,c$. 
We refer the readers to \cite{GL2018}   
for more details  and \cite{GHV2019} 
for the study on the minimum number of edges that occur in odd cycles.

\subsection{Spectral extremal graph problems}

Spectral graph theory aims to apply the eigenvalues of matrices associated
with graphs to find the structural information of graphs. 
Let $G$ be a simple graph on the vertex set $\{v_1,v_2,\ldots ,v_n\}$. The adjacency matrix of 
$G$ is defined as $A(G)=[a_{i,j}]_{i,j=1}^n$, 
where $a_{i,j}=1$ if $v_i$ and $v_j$ are adjacent, and $a_{i,j}=0$ otherwise. 
Let $\lambda (G)$ be the spectral radius 
of $G$, which is defined as the maximum modulus 
of eigenvalues of $A(G)$. 
Note that $A(G)$ is a non-negative matrix. 
By the Perron--Frobenius theorem, $\lambda (G)$ is the largest eigenvalue of $A(G)$. 
The study in this article mainly concentrates on the adjacency spectral radius.  

\medskip 
As mentioned before, 
the Tur\'{a}n type problem studies the maximum size of 
a graph that forbids certain subgraphs. 
In particular, one could wish to investigate the maximum possible spectral radius of the associated adjacency matrix of a graph that does not contain certain subgraphs.  
The interplay between these two areas above
is called the spectral Tur\'{a}n-type problem.   
One of the famous results of this type was obtained 
in 1986 by Wilf \cite{Wil1986} who showed that 
every graph $G$ on $n$ vertices with 
$\lambda (G) > (1- {1}/{r})n$ contains  
a clique $K_{r+1}$. 
This spectral version generalized the classical Tur\'{a}n theorem  
by invoking the fact  $\lambda (G) \ge 2m/n$. 
It is worth emphasizing that spectral Tur\'{a}n problems have been receiving considerable 
attention in the last two decades 
and it is still an attractive topic; see, e.g., 
\cite{Wil1986,Niki2002cpc,Niki2007laa2,LP2022second} for graphs with no cliques, \cite{BN2007jctb,LNW2021,ZhangST2024,ELW2024} for a conjecture of Bollob\'{a}s and Nikiforov, 
\cite{LNW2021,ZS2022dm,LFP2023-solution} for non-bipartite triangle-free graphs, 
\cite{TT2017,LN2021outplanar} for planar graphs and outerplanar graphs, 
\cite{Niki2009cpc} for a spectral Erd\H{o}s--Stone--Bollob\'{a}s theorem,  
\cite{Niki2009jgt} for the spectral stability theorem, 
\cite{CDT2023-even-cycle,LZS2024} 
for spectral problems on cycles, 
\cite{CDT2023} for a spectral Erd\H{o}s--S\'{o}s theorem, 
\cite{FLSZ2024} for some specific trees, 
\cite{ZL2022} for a spectral Erd\H{o}s--P\'{o}sa theorem, 
\cite{ZL2022jgt,Niki2021} for books and theta graphs, 
\cite{LN2023,Zhang2024} for cycles of consecutive lengths, 
\cite{Wang2022} for a spectral result 
on a class of graphs, and 
 \cite{Tait2019,ZL2022jctb} for graphs without 
 $K_t$-minors or $K_{s,t}$-minors. 

\medskip 
Although there has been a wealth of research results 
on the spectral extremal graph problems in recent years, 
there are \textbf{very few} conclusions on the problems 
of counting substructures in terms of spectral radius. 
The first result on this topic 
can even be traced back to a 
 result of Bollob\'{a}s 
and Nikiforov \cite{BN2007jctb} who 
showed that for every $n$-vertex graph $G$ and $r\ge 2$, 
the number of cliques of order $r+1$ satisfies 
\[  k_{r+1}(G) \ge 
\left( \frac{\lambda (G)}{n} - 1 + \frac{1}{r} \right) 
\frac{r(r-1)}{r+1} \left( \frac{n}{r}\right)^{r+1}.  \]

In 2023, Ning and Zhai \cite{NZ2021} 
studied the spectral saturation on triangles. 
A result of Erd\H{o}s and Rademacher states that 
every $n$-vertex graph $G$ with $e(G) > e(T_{n,2})$ contains at least 
$\lfloor \frac{n}{2} \rfloor$ triangles.  Correspondingly, 
it is natural to consider the  spectral version:  
if $G$ is a graph  with 
$ \lambda (G)> \lambda (T_{n,2})$, 
does $G$ have at least  $\lfloor \frac{n}{2}\rfloor$ triangles?    
Unfortunately, this result is not true. 
Let $K_{a,b}^+$ be the graph obtained from the complete bipartite graph $K_{a,b}$ by adding an edge to 
the vertex set of size $a$. 
For even $n$, we take $a=\frac{n}{2} +1$ and $b=\frac{n}{2}-1$. 
One can  verify that $\lambda (K_{\frac{n}{2}+1,\frac{n}{2}-1}^+) > \lambda (T_{n,2})$,   while  $K_{\frac{n}{2}+1,\frac{n}{2}-1}^+$ has exactly $\frac{n}{2}-1$ triangles.   
Recently, Ning and Zhai \cite{NZ2021} provided 
the following tight bound.

\begin{theorem}[Ning--Zhai, 2023] \label{thmNZ2021}
If $G$ is an $n$-vertex graph  with 
\[   \lambda (G) \ge \lambda (T_{n,2}), \]  
then  $G$ has at least 
$ \left\lfloor \frac{n}{2}\right\rfloor -1 $ triangles, 
unless $G$ is the bipartite Tur\'{a}n graph $T_{n,2}$.  
\end{theorem}

\section{Main results}

\subsection{Spectral radius vs triangular edges}

In the sequel, we shall put our attention on the extremal graph 
 problems 
concerning the spectral supersaturation. 
Specifically, 
we shall present a tight bound  on 
the number of triangular edges 
in a graph with spectral radius larger than that of $T_{n,2}$. Hence, we prove a spectral version of 
the result of Erd\H{o}s, Faudree and Rousseau.  

\begin{theorem}  \label{thm-main}
Let $G$ be a graph with $n\ge 5432$ vertices and 
\[ \lambda (G) \ge \lambda (T_{n,2}).\]  
Then $G$ has at least $2 \lfloor \frac{n}{2} \rfloor -1$ 
triangular edges, unless  $G=T_{n,2}$. 
\end{theorem}

The spectral condition in Theorem \ref{thm-main}
 is easier to satisfy than the edge-condition in Theorem \ref{thm-EFR}. 
Namely, if a graph $G$ satisfies $e(G)> e(T_{n,2})$, 
then $\lambda (G) > \lambda (T_{n,2})$.  
This observation can be guaranteed by 
$\lambda (G) \ge 2e(G) /n$. Nevertheless, 
there are many graphs with $\lambda (G)> \lambda (T_{n,2})$ 
but $e(G)< e(T_{n,2})$. Let $S_{n,k}$ be the split graph, which is the join of a clique of size $k$ with an independent set of size $n-k$. Taking $k=n/5$, we can verify that 
$S_{n,k}$ is a required example. 
A few words regarding the tightness of Theorem \ref{thm-main}  
are due. 
We show in next section that 
there exist three graphs $G$ such that $\lambda (G)>
\lambda (T_{n,2})$ and $G$ has exactly 
$2 \lfloor \frac{n}{2} \rfloor -1$ triangular edges, 
which implies the bound in Theorem \ref{thm-main}  is tight.   
 
  \medskip 
It is reasonable to reach 
such a difference between the results in 
Theorems \ref{thm-EFR} and \ref{thm-main}. 
Note that if $e(G)> 
e(T_{n,2})$, then $e(G) \ge e(T_{n,2}) +1$ holds  immediately. 
While, if  $\lambda (G) > \lambda (T_{n,2})$ holds, then there are many graphs with 
 $\lambda (G)$ very close to $\lambda (T_{n,2})$ and $e(G) = e(T_{n,2})$; see, e.g., the graphs in Figure \ref{fig-S2T1}. 
Roughly speaking, the spectral radii of graphs 
are distributed more compactly.  
Motivated by this observation, 
Li, Lu and Peng \cite{LLP2024-AAM} proposed a spectral conjecture on Mubayi's result \cite{Mub2010} and showed a spectral version of the Erd\H{o}s--Rademacher theorem.  Next, we are going to  provide a variant of Theorem \ref{thm-main}. We shall establish a spectral condition 
corresponding to the edge condition 
$e(G) \ge e(T_{n,2}) +1$. 
Recall that $K_{\lceil \frac{n}{2}\rceil , \lfloor \frac{n}{2}\rfloor}^+$ is the graph obtained from the complete bipartite graph 
$K_{\lceil \frac{n}{2}\rceil , \lfloor \frac{n}{2}\rfloor}$ by adding an edge to the vertex part of size $\lceil \frac{n}{2}\rceil$. 

\begin{theorem}   \label{thm-main3}
Let $G$ be a graph on $n\ge 5432$ vertices with  
\[ \lambda (G) \ge \lambda (K_{\lceil \frac{n}{2}\rceil , \lfloor \frac{n}{2}\rfloor}^+).\]  
Then $G$ has at least $2 \lfloor \frac{n}{2} \rfloor +1$ 
triangular edges, with equality if and only if 
$G=K_{\lceil \frac{n}{2}\rceil , \lfloor \frac{n}{2}\rfloor}^+$. 
\end{theorem}

\subsection{Our approach and applications}

\noindent 
{\bf Our approach.} 
Our proofs of Theorems \ref{thm-main} 
and \ref{thm-main3} are quite different from 
that of Theorem \ref{thmNZ2021}. 
It is a classical spectral method 
to use the Perron eigenvector together with the walks of length two  
to deduce the structural properties of spectral extremal graphs; see, e.g., \cite{TT2017,Tait2019,CDT2023,LN2021outplanar,NZ2021,ZL2022jctb}. 
However, applying this spectral method turns out to be difficult for graphs with much more triangles or triangular edges. 
The key ingredient in our proof attributes to 
a supersaturation-stability result  (Theorem \ref{thm-far-bipartite}), which 
roughly says that if a graph is far from being bipartite, 
then it contains a large number of triangles. 
This result may be of independent interest. 
Although we used the stability method, 
we only need a weak bound $n\ge 5432$ exactly\footnote{It seems possible to obtain a slightly better bound. To avoid unnecessary and tedious calculations, we did not attempt to get the best bound on the order of the graph in our proof. }, 
instead of the strong condition that $n$ is sufficiently large. 
Apart from the supersaturation-stability, 
another technique used in this paper is 
a spectral technique developed by  
Cioab\u{a},  Feng, Tait and  Zhang \cite{CFTZ20}; 
see, e.g., \cite{LP2021,DKLNTW2021,Wang2022} for recent results. Furthermore, we will obtain some approximately 
structural results that describe the almost-extremal graphs with 
large spectral radius and few triangular edges.

\medskip 
\noindent 
{\bf Applications.}  
With additional efforts, 
 the method used in the proof of Theorems \ref{thm-main} 
 and \ref{thm-main3} could possibly be applied to treat some spectral extremal  problems in which the desired extremal graph contains a small number of triangles. 
Incidentally, an upper bound on the number of triangular edges eventually leads to a restriction on the number of triangles. 
In particular, 
we shall present four quick applications of our method.  
(i) The first application gives a short proof of a conjecture of Erd\H{o}s, which asserts that every $n$-vertex graph with more than $n^2/4$ edges contains more than $n/6$ triangles sharing a common edge;   
(ii) The second application allows us to simplify the proof of the main result 
of \cite{CFTZ20}, and it can also improve the bound on the order of graphs, 
which was previously obtained from the celebrated triangle removal lemma;   
(iii) The third application is to deduce the classical stability result on odd cycles. 
Our approach can get rid of the use of the Erd\H{o}s--Stone--Simonovits theorem, and it yields more explicit parameters;    
(iv) The last application  provides an alternative proof of Theorem \ref{thmNZ2021} and
gives the complete characterization of the spectral extremal graphs of Theorem \ref{thmNZ2021}. 
We postpone the detailed discussions to Subsection \ref{sec4-3}.

\medskip

\noindent 
{\bf Organization.}  
 In Section \ref{sec3}, we shall present 
 some computations on the spectral radius of the expected extremal graphs. As mentioned above, these graphs 
 reveal that the bound in Theorem \ref{thm-main} cannot be improved. 
 Moreover, we will show the spectral version of 
 the supersaturation for triangular edges (Propositions \ref{super-1} and \ref{thm-m}). 
 In Section \ref{sec4}, 
one of the key ideas in this paper, i.e., 
the supersaturation-stability (Theorems \ref{thm-CFS-2020} and \ref{thm-far-bipartite}),
 will be introduced. 
 As indicated above,  some applications 
of the supersaturation-stability method will be presented in this section.  
In Sections \ref{sec5} and \ref{sec6}, 
we will present the detailed proofs of Theorems \ref{thm-main} and 
 \ref{thm-main3}, respectively.  
After proving our results, we propose some related spectral extremal problems  involving 
the edges that occur in cliques or odd cycles.

\medskip 
\noindent 
{\bf Notation.} 
We usually write $G=(V,E)$ for a simple graph with vertex set 
$V=\{v_1,\ldots ,v_n\}$ and edge set $E=\{e_1,\ldots ,e_m\}$, where we admit $n=|V|$ and $m=|E|$. 
If $S\subseteq V$ is a subset of the vertex set, then 
$G[S]$ denotes the subgraph of $G$ induced by $S$, 
i.e., the graph on $S$ whose edges are those edges of $G$ 
with both endpoints in $S$. By convention, we denote $e(S)=e(G[S])$.  
We will write $G[S,T]$ for the induced 
subgraph of $G$
whose edges have one endpoint in $S$ and the other in $T$,  and similarly, we write $e(S,T)$ for the number of edges 
of $G[S,T]$.  
Let $N(v)$ be the set of vertices adjacent to a vertex $v$ and 
let $d(v)=|N(v)|$. 
Moreover, we denote $N_S(v)= N(v) \cap S$ 
and $d_S(v)=|N_S(v)|$ for simplicity. 
 We will write $t(G)$ for the number of triangles of $G$. 
For an integer $p\ge 3$, we write $k_p(G)$ for the number of cliques of order $p$  in $G$.

\section{Preliminaries}
\label{sec3}

\subsection{Computations for extremal graphs}

We will show that  Theorem \ref{thm-main} is the best possible. 
Recall that $K_{a,b}^+$ denotes the graph obtained from 
a complete bipartite graph $K_{a,b}$ by adding an edge to 
the vertex part of size $a$. 
The following three graphs have spectral radii larger than 
$\lambda (T_{n,2})$ and contain exactly 
$2 \lfloor n/2 \rfloor -1$ triangular edges. 
Moreover, these graphs have exactly $\lfloor n^2/4\rfloor$ 
edges.

 \begin{figure}[H]
\centering
\includegraphics[scale=0.85]{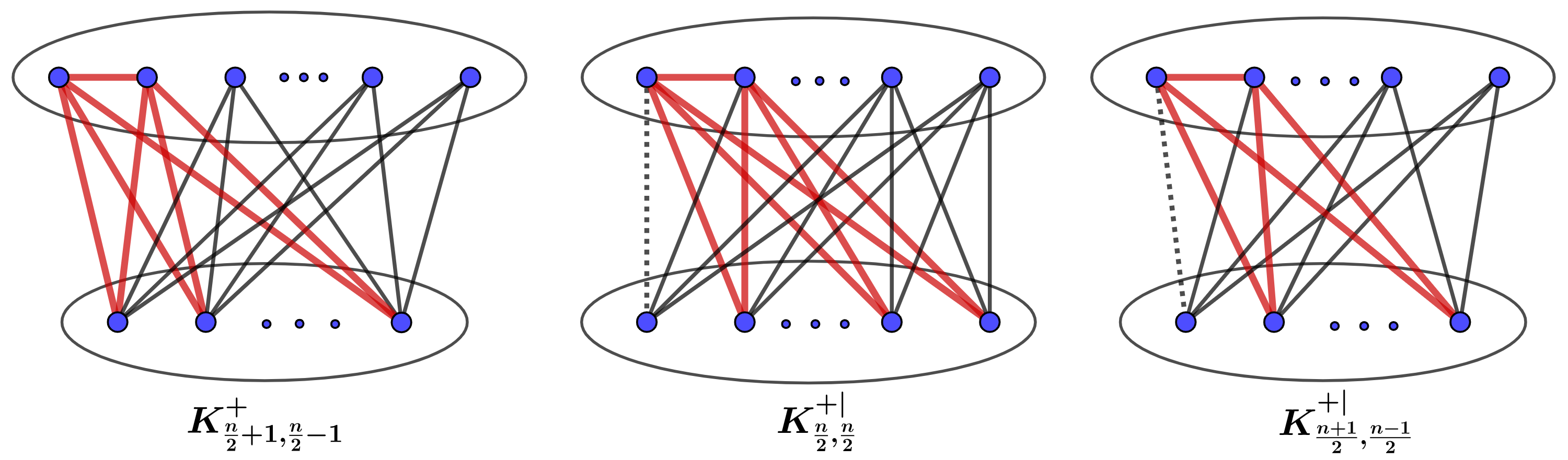} 
\caption{The graphs $K_{\frac{n}{2}+1,\frac{n}{2}-1}^+,
K_{\frac{n}{2} ,  \frac{n}{2} }^{+|}$ and 
$K_{ \frac{n+1}{2},\frac{n-1}{2}}^{+|}$.}
\label{fig-S2T1}
\end{figure}

\begin{lemma} \label{lem-pm-1} 
If $n\ge 4$ is even, then 
\[  \lambda (K_{\frac{n}{2}+1,\frac{n}{2}-1}^+) > \lambda (T_{n,2}). \] 
\end{lemma}

\begin{proof}
Let $\mathbf{x}=(x_1,x_2,\ldots ,x_n)^T$ be a Perron eigenvector corresponding to $\lambda(K_{\frac{n}{2} +1, \frac{n}{2}-1}^+)$.
We partition the vertex set of $ K_{\frac{n}{2}+1, \frac{n}{2}-1}^+$ as $\Pi$:
\[   V(K_{\frac{n}{2}+1, \frac{n}{2}-1}^+ )=
X_1 \cup X_2 \cup Y, \]
 where $X_1=\{u_1,u_2\}$ forms an edge,
$X_1\cup X_2$ and $Y$ are vertex sets of $K_{\frac{n}{2}+1, \frac{n}{2}-1}$ with $|X_1|+ |X_2|= \frac{n}{2}+1$ and $|Y|=\frac{n}{2} -1$.
By comparing the neighborhoods, we can see that $x_{u_1}=x_{u_2}$,
all coordinates of the vector $\mathbf{x}$
corresponding to vertices of $X_2$  are equal
(the coordinates of vertices of $Y$ are equal).
Without loss of generality, we may assume that
$x_{u_1}=x_{u_2}=x$, $x_u=y$ for each $u\in X_2$, and $x_v=z$
for each $v\in Y$. Then
\[  \begin{cases}
\lambda x= x + (\frac{n}{2}-1) z, \\
\lambda y = (\frac{n}{2} -1)z, \\
\lambda z = 2x + (\frac{n}{2}-1)y.
\end{cases} \]
Thus, $\lambda(K_{\frac{n}{2}+1, \frac{n}{2}-1}^+)$ is the largest eigenvalue of
\[  B_{\Pi} =  \begin{bmatrix}
1 & 0 & \frac{n}{2}-1 \\
0 & 0 & \frac{n}{2}-1 \\
2 & \frac{n}{2}-1 & 0
\end{bmatrix}.  \] 
Upon computation, it follows that $\lambda (K_{\frac{n}{2}+1,\frac{n}{2}-1}^+) $ is
the largest root of 
\[ f_1(x)=  \mathrm{det}(xI_3- B_{\Pi}) = x^3 - x^2 + x - {(n^2x)}/{4}  + {n^2}/{4} - n + 1 .\]
Since $f_1(\frac{n}{2})= 1-\frac{n}{2} <0$ for every $n\ge 4$, 
we have $\lambda (K_{\frac{n}{2}+1,\frac{n}{2}-1}^+) > \lambda (T_{n,2})=\frac{n}{2}$. 
\end{proof}

We point out that the partition $\Pi$ is an equitable partition\footnote{Given a graph $G$, the vertex partition $\Pi: V(G)=V_1\cup V_2 \cup \cdots \cup V_k$ is called an {\it equitable partition} if, for each $u\in V_i$, $|N(u)\cap V_j | = b_{i,j}$ is a constant depending only on $i,j \,(1\le i, j\le k)$.} and $B_{\Pi}$ is called the {\it quotient matrix} of $\Pi$. 
It is well-known \cite{CRS2010} that 
 the spectral radius of a graph $G$ 
 is equal to the the largest eigenvalue of the quotient matrix $B_{\Pi}$ corresponding to the equitable partition $\Pi$.

\begin{lemma} \label{lem-equal-pm}
Let $G=K_{\frac{n}{2} ,  \frac{n}{2} }^{+|}$ be the graph obtained from
$K_{ \frac{n}{2},  \frac{n}{2}}$ by adding an edge $e_1$ to the part of size $ \frac{n}{2} $ and deleting an edge $e_2$ between two parts such that $e_2$ is incident to $e_1$. Then 
 \[  \lambda (K_{\frac{n}{2} ,  \frac{n}{2} }^{+|} )> 
\lambda (T_{n,2}). \] 
\end{lemma}

\begin{proof}
By a similar method
as used in the proof of Lemma \ref{lem-pm-1}, we obtain that $\lambda (K_{\frac{n}{2} ,  \frac{n}{2} }^{+|} )$
is the largest root of
\[  f_2(x)=
x^4 - {(n^2x^2)}/{4}  - (n-2)x +1+ {n^2}/{2} -2n. \]
One can check that $f_2(\frac{n}{2})=1-n<0$ and hence $\lambda (K_{\frac{n}{2} ,  \frac{n}{2} }^{+|} )>\frac{n}{2} =
\lambda (T_{n,2})$.
\end{proof}

\begin{lemma} \label{odd-n-pm-one}
If $n\ge 5$ is odd and 
$G=K_{\frac{n+1}{2} , \frac{n-1}{2}}^{+|}$ is the graph obtained from
$K_{ \frac{n+1}{2},  \frac{n-1}{2}}$ by adding an edge $e_1$ to the part of size $ \frac{n+1}{2} $ and deleting an edge $e_2$ between two parts such that $e_2$ is incident to $e_1$, then 
\[  \lambda (K_{ \frac{n+1}{2},\frac{n-1}{2}}^{+|})
>   \lambda (T_{n,2}). \]
\end{lemma}

\begin{proof}
By a similar calculation, we know that $\lambda
(K_{ \frac{n+1}{2}, \frac{n-1}{2}}^{+|} )$
is the largest root of
\[  f_3(x)=x^4 - {(n^2x^2)}/{4} +x^2/4 
- (n-3)x + {n^2}/{2} - 2n+ {3}/{2}. \]
We can verify that 
\[ f_3\Bigl(\frac{1}{2}\sqrt{n^2-1} \Bigr) 
= \frac{1}{2}(n-3)\left( n-1-\sqrt{n^2-1} \right)<0, \] 
which implies $\lambda (K_{ \frac{n+1}{2},\frac{n-1}{2}}^{+|})
> \frac{1}{2}\sqrt{n^2-1} = \lambda (T_{n,2})$, as desired.
\end{proof}

The following lemma will be used in the proof of Theorem \ref{thm-main3}, and it 
provides 
a characterization of the spectral radius of the graph
$K_{ \lceil \frac{n}{2} \rceil, \lfloor \frac{n}{2} \rfloor}^+$.

\begin{lemma} \label{lem-21}
(a) If $n$ is even, then $\lambda(K_{\frac{n}{2}, \frac{n}{2}}^+)$
is the largest root of
\[  f(x)=x^3-x^2 - (n^2x)/4 + n^2/4 -n. \]
(b) If $n$ is odd, then $\lambda(K_{\frac{n+1}{2}, \frac{n-1}{2}}^+)$
is the largest root of
\[ g(x)=x^3 - x^2 + x /4 - (n^2x) /4 + n^2/4 - n + 3/4. \]
Consequently, for $n\ge 4$, we have
\begin{equation*} 
  \lambda^2 (K_{ \lceil \frac{n}{2} \rceil, \lfloor \frac{n}{2} \rfloor}^+) >\left\lfloor {n^2}/{4} \right\rfloor +2.
  \end{equation*}
\end{lemma}

\begin{proof}
By calculation, we can verify that for even $n$,
\[  f (\sqrt{{n^2}/{4 }+2} ) = \sqrt{n^2+8} -n -2<0, \]
and for every odd $n$,
\[  g (\sqrt{ {(n^2-1)}/{4} +2} ) = \sqrt{n^2+7} -n-1 <0. \]
So we get $\sqrt{\left\lfloor {n^2}/{4} \right\rfloor +2} < 
 \lambda (K_{ \lceil \frac{n}{2} \rceil, \lfloor \frac{n}{2} \rfloor}^+) $. This completes the proof.
\end{proof}

\subsection{Spectral supersaturation for triangular edges}

Recall that $t(G)$ denotes the number of triangles in a graph $G$. 
A special case of an aforementioned 
result of Bollob\'{a}s and Nikiforov \cite{BN2007jctb} 
states that 
\[  t(G) \ge \frac{n^2}{12}\left( \lambda -\frac{n}{2}\right). \]
From this inequality, we can obtain a spectral supersaturation for triangular edges. We denote by $\lambda (G)/n$ the spectral density of a graph $G$. 
Informally, once the spectral density of a graph exceeds that 
of the bipartite Tur\'{a}n graph, we can not only find $2\lfloor \frac{n}{2}\rfloor -1$ triangular edges, but in fact a large number of triangular edges with positive density, i.e., there are $\Omega(n^2)$ triangular edges.  
This gives a phase transition type result. 

\begin{proposition} \label{super-1}
If $\varepsilon >0$ and $G$ is a graph on $n$ 
vertices with 
\[  \lambda (G)\ge \frac{n}{2} + \varepsilon n, \]
then $G$ contains at least $32^{-1/3}\varepsilon^{2/3}n^2$ 
triangular edges. 
\end{proposition}

\begin{proof}
First of all, 
it follows from $\lambda \ge \frac{n}{2} + \varepsilon n$ 
that  
 \[  t(G) \ge \frac{n^2}{12}\left( \lambda -\frac{n}{2}\right) 
 \ge \frac{\varepsilon}{12}n^3. \]  
  Let $m'$ be the number of triangular edges of $G$,  
and  let $G'$ be the subgraph of $G$ whose edges consist of all the triangular edges of $G$.  Clearly, we have 
$t(G)=t(G')$. 
Applying the Kruskal--Katona theorem (see, e.g., \cite[page 305]{Bollobas78}), we get  
$ t(G') \le  \frac{\sqrt{2}}{3}(m')^{3/2}$, 
which implies  $m' \ge 32^{-1/3}\varepsilon^{2/3}n^2$. So 
$G$ has at least $32^{-1/3}\varepsilon^{2/3}n^2$ 
triangular edges. 
\end{proof}

In our proofs of Theorems \ref{thm-main} 
and \ref{thm-main3}, 
we need to use the following lemma, 
which counts the number of triangles in terms of the 
spectral radius and the size of a graph. 

\begin{lemma}[See \cite{BN2007jctb,CFTZ20,NZ2021}] 
\label{thm-BN-CFTZ-NZ}
Let $G$ be a graph with $m$ edges. Then 
\begin{equation*} \label{eq-spectral-super-triangle}
t(G) \ge \frac{\lambda \bigl(\lambda^2 - m\bigr)}{3}. 
  \end{equation*}
  The equality holds if and only if $G$ is a complete bipartite graph. 
\end{lemma}

The inequality can be written as the following versions: 
\begin{equation*} 
   \lambda^3 \le 3t + m \lambda 
\quad \Leftrightarrow \quad {t\ge \frac{1}{3}\lambda (\lambda^2 - m)  } 
\quad \Leftrightarrow \quad m \ge \lambda^2- \frac{3t}{\lambda}. 
\end{equation*}
This inequality was firstly published by 
Bollob\'{a}s and Nikiforov  as a special case 
of their result \cite[Theorem 1]{BN2007jctb}, 
and it was independently proved by Cioab\u{a},  Feng, 
Tait and Zhang \cite{CFTZ20}.  
The case of equality was characterized by Ning and Zhai \cite{NZ2021}.

\medskip 
From Lemma \ref{thm-BN-CFTZ-NZ}, we can see that 
every graph with  $\lambda (G) > \sqrt{m}$ contains a triangle. 
Next, we show a spectral supersaturation 
on the number of triangular edges.  

\begin{proposition} \label{thm-m}
If  $\varepsilon >0$ and $G$ is a graph with $m$ edges and 
\[ \lambda (G) \ge (1+ \varepsilon ) \sqrt{m}, \]    
then $G$ contains more than $2^{1/3}\varepsilon^{2/3}m$ triangular edges. 
\end{proposition}

\begin{proof} 
Since $\lambda \ge (1+ \varepsilon ) \sqrt{m}$, 
Lemma \ref{thm-BN-CFTZ-NZ}
 implies 
 \[  t(G) \ge \frac{\lambda (\lambda^2-m)}{3} >  
  \frac{2\varepsilon }{3} m^{3/2}. \]  
  Let $G'$ be the subgraph of $G$ whose edges consist of all the triangular edges of $G$.  We denote $m'=e(G')$.   
By the Kruskal--Katona theorem (see \cite[page 305]{Bollobas78}), we have   
$ t(G') \le \frac{\sqrt{2}}{3} (m')^{3/2}$. Then we get $m' > 2^{1/3}\varepsilon^{2/3}m$, and 
$G$ has more than $2^{1/3}\varepsilon^{2/3}m$ 
triangular edges. 
\end{proof}

\section{The supersaturation-stability method}

\label{sec4}

\subsection{The Lov\'{a}sz--Simonovits stability}

To prove and generalize the Erd\H{o}s conjecture on 
triangle-supersaturated graphs, Lov\'{a}sz and Simonovits \cite{LS1975} 
proved a stability result, 
and a much more general theorem in \cite{LS1983}, 
 the simplest form of which is the following: 

\begin{theorem}[Lov\'{a}sz--Simonovits, 1975]
For any constant $C>0$, there exists an $\varepsilon >0$ 
such that if $|k|< \varepsilon n^2$ and $G$ is an $n$-vertex graph with 
$\lfloor n^2/4 \rfloor +k$ edges and fewer than 
$C|k| n$ triangles, then one can remove $O(|k|)$ edges from 
$G$ to get a bipartite graph. 
\end{theorem}

It was shown in \cite{LS1983}  that if $G$ is an $n$-vertex graph with 
$e(G)=(1-\frac{1}{x})\frac{n^2}{2}$ edges, where $x > 1$ is a real number, 
then for any integer $p\le x+1$, 
the number of $p$-cliques satisfies $k_p(G)\ge {x \choose p} (\frac{n}{x})^p$; 
see, e.g., \cite[p. 449]{Lov1979} for a detailed proof.  
In the following, we introduce a more general theorem on stability.  
Let $T_{n,p}$ denote the $p$-partite Tur\'{a}n graph on $n$ vertices, that is,  $T_{n,p}$ is a complete $p$-partite graph 
 whose parts have sizes as equal as possible.  

\begin{theorem}[Lov\'{a}sz--Simonovits, 1983]
Let $C>0$ be an arbitrary constant. 
There exist constants $\delta >0$ and $C' >0$ such that 
if $1\le k < \delta n^2$ and $G$ is an $n$-vertex graph with 
$e(G)=(1-\frac{1}{x})\frac{n^2}{2}$, and $p\le x+1$ is an integer 
satisfying $e(G)=e(T_{n,p}) +k$  and 
\[ k_p(G) < {x \choose p} \left( \frac{n}{x}\right)^p + Ckn^{p-2}, \]
then $G$ can be made $\lfloor x \rfloor$-partite 
by removing at most $C'k$ edges. 
\end{theorem}

The application of the Lov\'{a}sz--Simonovits stability can be replaced here by an easy application 
of the graph removal lemma \cite{CF2013} 
and the Erd\H{o}s--Simonovits stability \cite{Sim1966}. 
The former result states that for every $\varepsilon >0$ and graph $H$ on $h$ vertices, there exists $\delta =\delta (H,\varepsilon)>0$ such that every $n$-vertex graph with at most $\delta n^h$ copies of $H$ can be made $H$-free by removing at most $\varepsilon n^2$ edges. This result was initially proved using the Szemerédi Regularity Lemma, i.e., the graph regularity method. The latter result says that for every $\varepsilon >0$ and graph $H$ with $\chi (H)=r+1\ge 3$, there exist $n_0$ and $\delta >0$ such that if $G$ is an $H$-free graph on $n\ge n_0$ vertices with $e(G)\ge (1- \frac{1}{r}- \delta)\frac{n^2}{2}$, then $G$ can be made $r$-partite by removing at most $\varepsilon n^2$ edges; see \cite{Fur2015} for an alternative proof. 
 For completeness, we present the following supersaturation-stability theorem. 
 In addition, we refer the readers to \cite{ADGS2015,CFS2020, FHW2021} 
for some similar applications 
on extremal set theory and Ramsey theory.  

\begin{theorem} \label{thm-CFS-2020}
For any $\varepsilon >0$ and  $r\ge 2$, 
there exist $\eta >0, \delta >0$ and $n_0\in \mathbb{N}$ such that if $G$ 
is a graph on $n\ge n_0$ vertices with at most 
$\eta n^{r+1}$ copies of $K_{r+1}$ and 
\[  e(G) \ge \left(1-\frac{1}{r} -\delta \right)\frac{n^2}{2}, \]  
then $G$ can be made $r$-partite by removing at most 
$\varepsilon n^2$ edges. 
\end{theorem} 

\begin{proof}
The graph removal lemma 
 allows us to pass to a $K_{r+1}$-free subgraph 
$G'$ of $G$ which still has very many edges. At this point, 
we can apply the standard stability theorem 
to deduce that $G'$ is nearly $r$-partite. 
Since we deleted few edges to go from $G$ to $G'$, 
we must also have that $G$ is nearly $r$-partite. 
\end{proof}

Although such an analogue can easily be obtained via the graph removal lemma, this gives 
bounds which are far from sufficient for our purposes. 
In the next subsection, we shall give a more efficient stability result 
so that we can calculate some explicit constants.

\subsection{A generalized Moon--Moser inequality}

First of all, 
we shall present a result of Moon and Moser \cite{MM1962}, 
which counts the minimum number of triangles in a graph with given order and size. 
Alternative proofs can also be found in \cite[p. 297]{Bollobas78} and \cite[p. 443]{Lov1979}. 

\begin{theorem}[Moon--Moser, 1962] \label{thm-MM-k3}
Let $G$ be a graph on $n$ vertices with $m$ edges. 
Then  
\[ t(G) \ge \frac{4m}{3n}\left(m- \frac{n^2}{4} \right), \]
where the equality holds if and only if $G=T_{n,r}$ with $r$ dividing $n$.  
\end{theorem}

We illustrate that the Moon--Moser theorem 
implies a supersaturation on triangles for graph with more than 
$n^2/4$ edges. 
For example,  if $G$ has at least ${n^2}/{4} +1$ edges, 
then it contains at least ${n}/{3}$ triangles.  
This result is slightly weaker than 
the Erd\H{o}s--Rademacher theorem. 
Moreover, the Moon--Moser theorem yields that 
if $\varepsilon >0$ and $G$ has  at least 
$ {n^2}/{4} + \varepsilon n^2$ edges,  
then $G$ contains more than ${\varepsilon} n^3/3$  triangles.  
In what follows, we shall show a generalization 
for graphs with less than $n^2/4$ edges. 

\medskip 
We say that a graph $G$ is  {\it $t$-far from being bipartite} 
if  $G'$ is not bipartite 
for every subgraph $G'$ of $G$ with 
$e(G')>  e(G) -t$, where $t$ is a positive real number.  
In other words, if $G$ is $t$-far from being bipartite, 
then no matter how we delete less than $t$ edges from $G$, 
the resulting graph is not bipartite. 
Equivalently, we must remove at least $t$ edges from 
$G$ to make it bipartite.  
It is well-known that every graph $G$ contains a bipartite subgraph $H$ 
with $e(H)\ge  e(G)/2$.  
From this  observation, we know that if  $G$ is said to be $t$-far from being bipartite, 
then we always admit the natural condition $t\le e(G)/2$.  

\medskip

Next, we present a counting result, 
which comes from the work of  Balogh, Bushaw, Collares, Liu, Morris 
and  Sharifzadeh \cite{BBCLMS2017}  during the study on the typical structure of  graphs with no large cliques.  
This result allows us to avoid the use of the triangle removal lemma 
or the Erd\H{o}s--Stone--Simonovits theorem, 
so that we could obtain a better bound on the order of the extremal graphs.  
This will be explained in the forthcoming Subsection \ref{sec4-3}.

\begin{theorem}[See \cite{BBCLMS2017}] \label{thm-far-bipartite}  
Let $G$ be a graph on $n$ vertices with $m$ edges. 
 If $t> 0$ and $G$ is $t$-far from being bipartite, then 
\[ t(G) \ge \frac{n}{6}\left( m +t-\frac{n^2}{4}\right).  \] 
\end{theorem}  

We provide a detailed proof for completeness. 
This result can be proved by applying a similar argument 
due to  Sudakov \cite{Sud2007}, F\"{u}redi \cite{Fur2015} and Conlon, Fox and Sudakov \cite{CFS2020-Siam}.

\begin{proof} 
 For each $v \in V(G)$, we denote $N_v=N(v)$ 
 and $N_v^c=V(G)\setminus N(v)$.  
Since $G$ is $t$-far from being bipartite, it follows  that 
for every $v\in V(G)$, 
\[  e(N_v) + e(N_v^c) \ge t . \] 
On the one hand, we have 
\[  \sum_{w\in N_v^c} d(w) = 2e(N_v^c) + e(N_v^c,N_v) 
= e(N_v^c) + m - e(N_v) 
\ge m+t-2e(N_v).  
  \]
Summing over all vertices $v\in V(G)$ yields 
  \[  \sum_{v\in V(G)} \sum_{w\in N_v^c} d(w) 
  \ge mn +nt - 2 \sum_{v\in V(G)} e(N_v) = 
  mn +nt - 6t(G), \]
  where we used the fact $\sum_{v\in V(G)}e(N_v) = 3t(G)$. 
 On the other hand, we get  
\[ \sum_{v\in V(G)} \sum_{w\in N_v^c} d(w)= 
\sum_{v\in V(G)} \left( 2m - \sum_{w\in N_v}d(w) \right) 
= 2mn - \sum_{w\in V(G)} d^2(w). \]
Combining these two inequalities, 
we obtain 
\[  6t(G) 
\ge nt- mn + \sum_{w\in V(G)} d^2(w) 
\ge nt-mn + \frac{4m^2}{n}. \]
Observe that  
${4m^2}/{n} \ge 2mn - {n^3}/{4}$. 
The required bound holds immediately. 
\end{proof}

The following result is a direct consequence of Theorem \ref{thm-far-bipartite}. 

\begin{corollary}  \label{coro-bip}
If $G$ is an $n$-vertex graph with $m = {n^2}/{4} -q$ edges, where $q\in \mathbb{Z}$, 
and $G$ has at most $t$ triangles, then we can remove at most  
${6t}/{n} +q$ edges to make it bipartite, so $G$ has a bipartite 
subgraph with size at least ${n^2}/{4}  - {6t}/{n} -2q$. 
\end{corollary}

This corollary can also be deduced from 
the result of Sudakov \cite[Lemma 2.3]{Sud2007}.

\subsection{Applications of the supersaturation-stability}

\label{sec4-3}

There are several advantages in the supersaturation-stability method. 
As promised, we now present four quick applications of this method. 
In our framework, 
we will take advantage of the results in Theorem \ref{thm-far-bipartite} or Corollary \ref{coro-bip} with some appropriate structural analysis.

\subsubsection{The Erd\H{o}s conjecture involving the booksize}

Recall that a book of size $t$ consists of $t$ triangles 
that share a common edge. 
The study of bounding the largest size of a book in a graph was initially investigated by 
Erd\H{o}s \cite{Erd1962a} who proved that every $n$-vertex graph  with at least 
$\lfloor n^2/4\rfloor +1$ edges contains a book of size 
$n/6 - O(1)$,  and conjectured that the term $O(1)$ can be removed. This conjecture was later 
proved by Edwards (unpublished, see \cite[Lemma 4]{EFR1992}) 
and independently by Khad\v{z}iivanov and Nikiforov \cite{KN1979} (unavailable, see \cite{BN2005}). 
Unfortunately, neither of the two original references can be found.  
Here, we show that  Theorem \ref{thm-far-bipartite} 
can easily confirm the Erd\H{o}s conjecture.  
More precisely, we can use Theorem \ref{thm-far-bipartite}  
to prove that every graph $G$ 
on $n$ vertices with more than $n^2/4$ edges contains a book of size greater than $n/6$. 
Indeed, assume that $G$ has exactly $t$ triangles, then 
Theorem \ref{thm-far-bipartite}  yields that 
$G$ is not ${6t}/{n}$-far from being bipartite. 
Specifically, one can remove less than ${6t}/{n}$ edges from $G$ to destroy all $t$ triangles. So one of these edges must be contained in more than $n/6$ triangles, as needed. 
For more related results, we refer the readers to 
\cite{BN2005,Niki2021,ZL2022jgt} and the references therein.

\subsubsection{Eliminating the use of triangle removal lemma}

In 2020, Cioab\u{a},  Feng, Tait and  Zhang \cite{CFTZ20} studied the spectral extremal graphs 
of order $n$ for the friendship graph $F_k$ and 
 sufficiently large $n$, where $F_k$ is the graph that consists of 
 $k$ triangles sharing a vertex. 
Their proof uses the Ruzsa--Szemer\'{e}di triangle removal lemma, 
which settles the problem in the case where $k$ is fixed,  and
the result is meaningless when $k$ is large and growth with $n$ 
(say, when $k\ge  \log n$).   
Using the supersaturation-stability method, 
instead of the triangle removal lemma, 
we can show that the main result in \cite{CFTZ20} is valid for every $k\le \frac{1}{21}n^{1/4}$. 
This considerably extends the range of $k$. 
The main ingredient is to prove the following lemma, 
which can substantially simplify the original proof. 

\begin{lemma} \label{lem-Fk}
If $G$ is an $F_k$-free graph on $n$ vertices
 and $\lambda (G)\ge {n}/{2}$, 
then 
\[  e(G)> \frac{n^2}{4} - 54k^2 \] 
 and there exists a vertex partition of $G$ as  
$V(G)=S\cup T$ such that 
\[  e(S) + e(T)< 108k^2. \]  
Moreover, we have 
\[  \frac{n}{2} -13k < |S|, |T| < \frac{n}{2} +13k \] 
 and 
\[  \frac{n}{2} - 56 k^2 < \delta (G) 
\le \lambda (G) \le \Delta (G) < \frac{n}{2} + 14k. \] 
\end{lemma}

\begin{proof}
A result due to Alon and Shikhelman \cite[Lemma 3.1]{AS2016} states that 
if $G$ is $F_k$-free, then $G$ has less than 
$(9k-15)(k+1)n < 9k^2n$ triangles. 
Using Lemma \ref{thm-BN-CFTZ-NZ}, 
we have $e(G) \ge \lambda^2- {(3t)}/{\lambda} \ge  
\lambda^2 - {(6t)}/{n} > n^2/4 - 54k^2$.  
Then it follows from Theorem \ref{thm-far-bipartite}  
that $G$ is not $108k^2$-far from being bipartite.  
Thus, we can remove less than $108k^2$ edges from $G$ to obtain 
a bipartite subgraph. Equivalently, there exists a vertex partition $V(G)=S\cup T$ such that 
$e(S) + e(T)< 108k^2$. Therefore, we get $e(S,T)\ge e(G)- 108k^2 
> n^2/4 - 162k^2$, which implies 
$n/2 - 13k < |S|,|T| < n/2 + 13k$. 
Furthermore, we have  $\delta (G)>  n/2 - 56k^2$.  
Otherwise, if $d(v)\le n/2 - 56k^2$ for some $v\in V(G)$, 
then $e(G\setminus \{v\})\ge n^2/4 - 54k^2 - (n/2 - 56k^2) 
> (n-1)^2/4 + k^2$, which leads to a copy of $F_k$ in $G\setminus \{v\}$, 
a contradiction. 
Since $\delta(G) > n/2 - 56k^2$, using the inclusion-exclusion principle, 
we can show that both $G[S]$ and $G[T]$ 
are $K_{1,k}$-free and $M_k$-free. Then 
 $\Delta (G)< (n/2 + 13k) + k \le n/2 + 14k$. 
\end{proof}

The key innovation in our argument 
is to exploit the supersaturation-stability. 
 Lemma \ref{lem-Fk} can have the same role as that from \cite[Lemma 15]{CFTZ20}. 
Consequently, we provide a new approach to simplifying  
 many  technical lemmas as stated in \cite{CFTZ20} 
 so that we can get rid of the use of triangle removal lemma and drop the condition requiring  $n$ to be sufficiently large.  
Note that Li, Lu and Peng  \cite{2022LLP} 
revisited the spectral extremal graph for the bowtie $F_2$ 
and showed a tight bound $n\ge 7$ in another different way. 
In addition, Lemma \ref{lem-Fk} can also be applied 
to the proof of a recent result due to Lin, Zhai and Zhao \cite[Theorem 7]{LZZ2022}.

\subsubsection{Concise stability result for odd cycles}

The classical stability of Erd\H{o}s and Simonovits says that 
for any $\varepsilon >0$ and any graph $F$ with $\chi (F)=r+1$, there exist $n_0$ and $\delta >0$ such that 
if $G$ is an $F$-free graph on $n\ge n_0$ vertices with 
$e(G)\ge (1- \frac{1}{r} - \delta )\frac{n^2}{2}$, then $G$ can be made $r$-partite by removing at most $\varepsilon n^2$ edges. Moreover,  F\"{u}redi \cite{Fur2015} proved that 
if $G$ is an $n$-vertex $K_{r+1}$-free graph with 
$e(G)\ge e(T_{n,r}) - t$ edges, then $G$ can be made $r$-partite by removing at most $t$ edges. 
This gives a concise dependency $\delta =2\varepsilon$. 
 The concise stability for cliques are well-studied in the past few years;  see \cite{RS2018,Liu2021,BCLLP2021,KRS2021} and references therein. 

We point out that 
the supersaturation-stability method may be utilized to get better bounds 
for treating the extremal problems 
on $C_{2k+1}$-free graphs or $kC_3$-free graphs. 
By applying Corollary \ref{coro-bip}, 
we can prove the following concise stability for odd cycles.

\begin{theorem}[Concise stability]
For every $k\ge 1$ and $0<\varepsilon < 1/2$, 
we denote $\delta := \varepsilon /2$ and 
$n_0:= 2k / \varepsilon$. 
If $G$ is a $C_{2k+1}$-free graph on $n\ge n_0$ vertices 
with $e(G)\ge ({1}/{4} - \delta )n^2$, 
then $G$ can be made bipartite by deleting at most $\varepsilon n^2$ edges. 
\end{theorem}

\begin{proof}
Since $G$ is $C_{2k+1}$-free, we know that $e(G)\le n^2/4$ 
for $n\ge 4k$. Note that 
$G[N(v)]$ is $P_{2k}$-free for each $v\in V(G)$.
Then $3t(G) = \sum_{v\in V} e(N(v))  \le \sum_{v\in V} kd(v) \le 2km \le \frac{1}{2}kn^2$, 
where the first inequality holds by the Erd\H{o}s--Gallai theorem. 
By Corollary \ref{coro-bip}, we can remove 
at most $6t/n +q \le kn + \delta n^2\le \varepsilon n^2$ 
edges to make $G$ bipartite. 
\end{proof}

The above proof gives a new short proof of the stability for odd cycles, but also presents  a {\it linear dependency} between $\delta$ and $\varepsilon$.  
However, the conventional proof for stability 
is based on applying the Erd\H{o}s--Stone--Simonovits theorem, which gives 
bad bounds on $\delta$ and $n_0$\footnote{We refer to Conlon's lecture note; see \url{http://www.its.caltech.edu/~dconlon/EGT12.pdf}}. 
Similarly, we can show the following concise stability for the 
spectral radius. 

\begin{theorem}
For every $k\ge 1$ and $\delta \ge 0$, if $G$ is a 
 $C_{2k+1}$-free graph on $n$ vertices with spectral radius 
 $\lambda (G)\ge {n}/{2} - \delta $, then 
$e(G)\ge n^2/4 - (\delta+2k)n$ and $G$ can be made bipartite by removing at most $(\delta +3k)n$ edges. 
\end{theorem}

\begin{proof}
Note that $3t(G)\le \frac{1}{2}kn^2$. 
Lemma \ref{thm-BN-CFTZ-NZ} implies 
$e(G)\ge \lambda^2- (3t)/ \lambda \ge n^2/4 - \delta n - 2kn$. Applying Corollary \ref{coro-bip}, we can remove at most $ \delta n + 3kn$ edges to make $G$ bipartite. 
\end{proof}

\subsubsection{An alternative proof of the Ning--Zhai theorem}

 Finally, 
we shall present the fourth application by giving  
an alternative new proof of Theorem \ref{thmNZ2021}. 
Our approach is completely different from the original proof in \cite{NZ2021}, 
and it is primarily based on the supersaturation-stability, 
while the original proof relies on the structural analysis of the extremal graph  by counting the $2$-walks starting from the largest entry of the Perron vector.  
Furthermore, our proof allows us to show that 
 the extremal graphs in Theorem \ref{thmNZ2021} are the same as those 
 in Theorem \ref{thm-main}; see Figure \ref{fig-S2T1}. 
In other words, we can determine 
all the extremal graphs $G$ satisfying
 $\lambda (G) > \lambda (T_{n,2})$ and 
 $t(G)= \left\lfloor {n}/{2}\right\rfloor -1$.  
To more clearly demonstrate the main ideas of our approach, we assume that $n\ge 36$ in order to avoid the tedious computation. Now, we briefly describe the main steps.

\begin{proof}[{\bf New proof of Theorem \ref{thmNZ2021}}]
Assume that $G$ is an $n$-vertex graph with $\lambda (G) \ge  
\lambda (T_{n,2})$ and $G\neq T_{n,2}$. 
Moreover, we assume further that $G$  has the minimum number of 
triangles. 
Then 
$  t(G)\le \left\lfloor \frac{n}{2}\right\rfloor -1\le \frac{n-2}{2}$.  
Note that  $\lambda (G)\ge \lambda (T_{n,2}) > \frac{n-1}{2}$. 
By Lemma \ref{thm-BN-CFTZ-NZ}, we get  
\[   e(G)\ge \lambda^2- \frac{3t}{\lambda} 
> \lambda^2 - \frac{6t}{n-1} \ge   
\left\lfloor \frac{n^2}{4} \right\rfloor - \frac{3(n-2)}{n-1} . \]  
Note that $e(G)$ must be an integer.  
Then 
\[ e(G)\ge \left\lfloor \frac{n^2}{4} \right\rfloor  -2. \] 
 If $G$ is  $6$-far from being bipartite, then Theorem \ref{thm-far-bipartite} implies that 
\[  t(G)\ge \frac{n}{6} \left( e(G) +6 - \frac{n^2}{4} \right) > \frac{n}{2}, \]  
 a contradiction. Thus, $G$ is not $6$-far from  being bipartite. Consequently, there is a partition of
the vertex set of $G$ as $V(G)=S\cup T$ such that $e(S)+e(T)< 6$. Then 
\[  e(S,T) = e(G) - e(S) - e(T)
\ge e(G) -5 \ge \left\lfloor \frac{n^2}{4} \right\rfloor -7. \]    
By the AM-GM inequality, we get 
\[  \left \lfloor \frac{n}{2} \right\rfloor -2 
\le |S|,|T| \le \left \lceil \frac{n}{2} \right\rceil +2. \] 
We say an edge is a class-edge of $G$ if the endpoints 
 of this edge are either both in $S$ or both in $T$. 
Similarly,  an edge is said to be a cross-edge 
 if it has one endpoint in $S$ and the other in $T$.  
Next, we claim that there is exactly one class-edge in $G$. Namely, 
\[ e(S) +e(T) =1. \]  
Otherwise, suppose  that $G$ has $s$ class-edges, where $2\le s\le 5$. 
Observe that each missing cross-edge between $S$ and $T$ 
is contained in at most $s$ triangles. Then for $n\ge 36$, 
we have $ t(G) \ge s ( \lfloor \frac{n}{2}\rfloor -2) - 7s > \lfloor \frac{n}{2}\rfloor -1$, 
a contradiction.  Thus, we conclude that 
$ e(S) + e(T)=1$.  
Using this claim, 
we can make a slight refinement as below:  
\[  e(S,T)= e(G)- 1 \ge \left\lfloor \frac{n^2}{4} \right\rfloor -3 \] 
and 
\[  \left \lfloor \frac{n}{2} \right\rfloor -1 
\le |S|,|T| \le \left \lceil \frac{n}{2} \right\rceil +1. \] 
Without loss of generality, we may assume that 
 $e(S)=1$ and $e(T)=0$. 
 Thus, $G$ is a subgraph of $K_{s,t}^+$ with  $s\in [\frac{n}{2} -1, \frac{n}{2} +1]$, and $G$ satisfies $\lambda (G)\ge \lambda (T_{n,2})$
and $t(G)\le \lfloor \frac{n}{2}\rfloor -1$. 
 Finally, using a simple argument, 
 we can compute that 
 \[ G\in \left\{ K_{\frac{n}{2}+1,\frac{n}{2}-1}^+,
K_{\frac{n}{2} ,  \frac{n}{2} }^{+|}, 
K_{ \frac{n+1}{2},\frac{n-1}{2}}^{+|} \right\}. \] 
For simplicity, we omit the tedious calculation, since a similar argument 
can be found in the remark after the proof of Theorem \ref{thm-main} in Section \ref{sec5}.
\end{proof}

\noindent 
{\bf Remark.} 
A theorem of Erd\H{o}s and Rademacher \cite{Erd1955,Erdos1964} states that 
if $e(G)> e(T_{n,2})$, then  $t(G)\ge \lfloor {n}/{2}\rfloor$.  
At first glance, the Erd\H{o}s--Rademacher theorem and Theorem \ref{thmNZ2021} seem incomparable. 
In the above proof, after determining the extremal graphs in 
 Theorem \ref{thmNZ2021}, we can show that 
 Theorem \ref{thmNZ2021} actually implies the Erd\H{o}s--Rademacher theorem. 
 Indeed, as long as $G$ is a graph with $e(G)> e(T_{n,2})$, 
 by the fact  $\lambda (G)\ge 2e(G)/n$,  we can get $\lambda (G)> \lambda (T_{n,2})$. 
Then Theorem \ref{thmNZ2021} gives $t(G)\ge \lfloor {n}/{2}\rfloor -1$, 
 while the graphs attaining the equality has exactly $\lfloor {n^2}/{4}\rfloor$ edges. Therefore, we have $t(G)\ge \lfloor {n}/{2}\rfloor $, as expected. 
It turns out to be meaningful to characterize  
the equality case of Theorem \ref{thmNZ2021}
in this sense.

\section{Proof of Theorem \ref{thm-main}}

\label{sec5}

Assume that $G$ is a graph of order $n$ with 
$\lambda (G) \ge \lambda (T_{n,2})$ and $G\neq T_{n,2}$, 
we need to prove that $G$ has 
at least  $2\lfloor n/2\rfloor -1$ triangular edges. 
Suppose on the contrary that 
$G$ has less than $2\lfloor n/2\rfloor -1$ triangular edges 
(This bound can be changed to $2\lfloor n/2\rfloor +1$ 
in order to adapt the proof of Theorem \ref{thm-main3}). 
Among such counterexamples, we choose $G$ as a graph 
with the maximum spectral radius.

\begin{lemma}\label{lem-partition}
There exists a vertex partition  $V(G)=S\cup T$ such that 
\[ e(S) + e(T)< 6\sqrt{n} \]
and 
$$e(S, T)> \frac{n^2}{4} - 9\sqrt{n}. $$
Furthermore, we have 
\[
\frac{n}{2} - 3n^{1/4} < |S|, |T| < 
\frac{n}{2} + 3n^{1/4}.
\]
\end{lemma}

\begin{proof}
Since $G$ has less than 
$n$ triangular edges, 
we know from the Kruskal--Katona theorem (see, e.g., \cite[page 305]{Bollobas78}) 
that $G$ has less than $\sqrt{2}n^{3/2}/3 < n^{3/2}/2$ triangles. 
Note that $\lambda (G)\ge \lambda (T_{n,2})= \sqrt{\lfloor n^2/4\rfloor} > {(n-1)}/{2}$.  Then Lemma \ref{thm-BN-CFTZ-NZ} implies 
\[  e(G) \ge \lambda^2 - \frac{6t}{n-1}> \frac{n^2}{4} - 3\sqrt{n}.  \]
We claim that $G$ is not $6\sqrt{n}$-far from being bipartite. 
Suppose in contrast that $G$ is $6\sqrt{n}$-far from being bipartite. Then Theorem \ref{thm-far-bipartite} 
implies that $G$ has at least ${n}/{6} 
({n^2}/{4} - 3\sqrt{n} + 6\sqrt{n} - {n^2}/{4}) 
= n^{3/2}/2$ triangles, a contradiction. Therefore, $G$ is not $6\sqrt{n}$-far from 
being bipartite. Namely, there exists a vertex partition of 
$G$ as $V(G)=S\cup T$ such that 
\begin{equation*}  
e(S) + e(T) < 6\sqrt{n}. 
\end{equation*}
 Consequently, we get 
\[ e(S,T)> e(G) - 6\sqrt{n} > \frac{n^2}{4} - 9\sqrt{n}. \] 
Without loss of generality, we may assume that 
$1\le |S| \le |T|$. Suppose on the contrary that 
 $|S| \le {n}/{2} - 3n^{1/4}$. Then by $|S| + |T|=n$, we have 
$|T| \ge {n}/{2}  + 3n^{1/4}$. It follows that 
$e(S,T)\le |S| |T| \le ({n}/{2} - 3n^{1/4})({n}/{2} + 3n^{1/4})= {n^2}/{4} - 9 n^{1/2}$, a contradiction. 
Thus, we obtain $|S| > {n}/{2} - 3n^{1/4}$ and 
 $|T| = n-|S| < {n}/{2} + 3n^{1/4}$, as required. 
        \end{proof}

Lemma \ref{lem-partition} guarantees that there exists a partition with  
        $e(S,T) > n^2/4 - 9\sqrt{n}$ 
        and $e(S) + e(T) < 6\sqrt{n}$. 
          Among such partitions, we may assume further that 
        $V(G)=S\cup T$ is a partition with  maximum cut, i.e., 
        the bipartite subgraph $G[S,T]$ has the maximum number of edges. 
        Next, we define two sets of ‘bad’ vertices of $G$. Namely, we define  
\[ L:= \left\{v\in V(G): d(v) \leq 
\left(\frac{1}{2}-\frac{1}{200} \right) n\right\} .\]  
        For a vertex $v\in V(G)$, let $d_S(v) = |N(v) \cap S|$ and $d_T(v) = |N(v) \cap T|$. We denote 
\[ W: = \left\{ v\in S: d_S(v) \geq  \frac{n}{140} \right\} \cup 
\left\{v \in T: d_T(v) \geq  \frac{n}{140} \right\}. \]
First of all,  we show that both $W$ and $L$ are small sets.

\begin{lemma}\label{Lupper}  
We have 
$|L| < 10.$ 
\end{lemma}

\begin{proof}
 Suppose that $|L| \ge 10$.
 Then let $L^{\prime}\subseteq L$ with $|L^{\prime}|=10$. 
 We consider the subgraph of $G$ obtained by deleting all the vertices of $L'$. 
It follows that
   \begin{eqnarray*}
e(G\setminus L^{\prime}) &\ge& e(G)-\sum_{v\in L^{\prime}}d(v)\\
&\ge& \frac{n^2}{4}- 3\sqrt{n} 
-10 \left(\frac{1}{2}-\frac{1}{200} \right)n\\
&\ge & \frac{(n-10)^2}{4}+25, 
     \end{eqnarray*}
where the last  inequality holds for $n\ge 5416$. 
By modifying the proof of Theorem~\ref{thm-EFR}, 
 we can see that 
 the subgraph $G\setminus L^{\prime}$ contains more than $n+1$ triangular edges, a contradiction 
 (In fact, a result of F\"{u}redi and Maleki \cite[Theorem 1.2]{FM2017} can indicate 
 more triangular edges in $G\setminus L'$). 
So we have $|L| < 10$. 
 \end{proof}

\begin{lemma}\label{W-Lenpty}
We have
$|W| < \frac{1680}{\sqrt{n}}.$
\end{lemma}

\begin{proof}
We denote $W_1=W\cap S$ and $W_2=W\cap T$. Then 
 $$2e(S) =\sum_{u\in S}d_{S}(u) \ge  \sum_{u\in W_1}d_S(u)\ge  \frac{n}{140} |W_1| $$ 
 and 
 \[  2e(T) = \sum_{u\in T}d_{T}(u)\ge  \sum_{u\in W_2}d_T(u)
 \ge \frac{n}{140} |W_2| . \]
  So we obtain 
  \begin{equation*} \label{WL-2}
  e(S)+e(T)\ge (|W_1|+|W_2|)\frac{ n}{280} 
  =\frac{|W|  n}{280}.
  \end{equation*}
 On the other hand, according to Lemma \ref{lem-partition}, we have 
 \[  e(S)+e(T)< 6\sqrt{n}. \] 
  Then we get ${|W| n}/{280}<  6\sqrt{n}$, 
that is,
$ |W|< {1680}/{\sqrt{n}}$, as needed. 
  \end{proof}

  We will also need the following inclusion-exclusion principle.

\begin{lemma}\label{set}
Let $A_1, A_2,\dots, A_k$ be $k$  finite sets. Then
\begin{equation*}\label{set1}
\left|\bigcap_{i=1}^k A_i \right| \ge \sum_{i=1}^k |A_i|-(k-1)\left|\bigcup_{i=1}^k A_i\right|.
\end{equation*}
\end{lemma}

  \begin{lemma} \label{lem-W-sub-L}
  We have $W \subseteq L$ and $|W| \le |L| < 10$. 
  \end{lemma}
  
  \begin{proof}
We shall prove that if $u\notin L$, then $u\notin W$. 
We denote $L_1=L\cap S$ and $L_2=L\cap T$. Without loss of generality, we may assume that $u\in S$ and 
   $u\notin L_1.$ Since  $S$ and $T$ form a maximum cut in $G$, 
   we claim that $d_T(u)\ge \frac{1}{2}d(u)$. 
   Otherwise, if $d_T(u)< \frac{1}{2}d(u)$, then 
   by $d(u)=d_S(u) + d_T(u)$, we have 
   $d_S(u) > d_T(u)$. Moving the vertex $u$ from 
   $S$ to $T$ yields a new vertex bipartition with more edges, 
   which contradicts with the maximality of $G[S,T]$. 
   So we must have $d_T(u)\ge \frac{1}{2}d(u)$. 
  On the other hand, 
  we have $d(u)> \left(\frac{1}{2}-\frac{1}{200} \right)n$ 
  since $u\not\in L$. Then 
   $$d_T(u)\ge \frac{1}{2}d(u)> 
   \left(\frac{1}{4}-\frac{1}{400} \right)n.$$
  Recall that 
  $|L|< 10$ and $|W| < 1680/ \sqrt{n}$, 
we have 
$  |S\setminus (W\cup L)| \approx \frac{n}{2}$. 
  We claim that $u$ has at most $7$ neighbors in 
  $S\setminus (W\cup L)$. 
  Indeed, suppose on the contrary that 
   $u$ is adjacent to $8$ vertices  $u_1, u_2,\ldots, u_8$ in $S\setminus (W\cup L)$.  Since $u_i\not\in L$,  we have $d(u_i)> \left(\frac{1}{2}-\frac{1}{200}\right )n$. Similarly, 
   we have $d_S(u_i)< \frac{n}{140}$ as $u_i\notin W$. So
  \[  d_T(u_i)=d(u_i)-d_S(u_i)> \left(\frac{1}{2}-\frac{1}{200} - \frac{1}{140} \right)n. \] 
   By Lemma~\ref{set}, we have
   \begin{eqnarray*}
&& \left|N_T(u)\cap N_T(u_1) \cap \cdots 
\cap N_T(u_8) \right|\\[2mm]
 &\ge&  |N_T(u)|+ |N_T(u_1)| + \cdots + |N_T(u_8)|  -8|T| \\
& > &\left (\frac{1}{4}-\frac{1}{400}\right)n+\left(\frac{1}{2}-\frac{1}{200} -\frac{1}{140} \right)n\cdot 8 - 8 \left( \frac{n}{2} + 3n^{1/4}\right)\\
 &>&\frac{n}{9}, 
\end{eqnarray*}
 where the last inequality holds for $n\ge 5191$.  
Let $B$ be the set of common neighbors of $u,u_1,\ldots ,u_8$ in $T$. 
Then $|B|>n/9$. 
Observe that for each vertex $v\in B$, the $vuu_i$ forms a triangle 
for each $1\le i\le 8$, so $vu,vu_i (1\le i \le 8)$ are triangular edges. That is to say, each vertex of $B$ 
is incident to at least $9$ triangular edges. 
This leads to more than $9|B| +8 > n +8$ triangular edges, a contradiction.
Therefore $u$ is adjacent to at most $7$ vertices in  $S\setminus (W\cup L)$. 
Recall that $|L|\le 9$ by Lemma \ref{Lupper}.  
Hence, for $n\ge 5432$, we have 
\[  d_S(u) \le |W|+|L|+ 7 
<  \frac{1680}{\sqrt{n}} + 16 < \frac{n}{140}. \] 
By definition, we get $u\notin W$.
This completes the proof. 
 \end{proof}

\begin{lemma}  \label{lem-one-edge}
We have 
$e(S\setminus L) \le 1$ and $e(T\setminus L) \le 1$. 
Consequently, 
there exist  independent sets
 $I_S\subseteq S\setminus L$  and $I_T\subseteq T\setminus L$ such that
$ |I_S|\ge  |S|-10$
   and 
$|I_T|\ge |T|-10 $. 
 \end{lemma}

\begin{proof}
Firstly, we show that $e(S\setminus L) \le 1$ and $e(T\setminus L) \le 1$. 
Suppose on the contrary that 
 $G[S\setminus L]$ contains two edges, say $e_1,e_2$.  
 We shall deduce a contradiction in two cases. 
 
If $e_1$ and $e_2$ are intersecting, then we assume that 
$e_1=\{u_1,u_2\}$ and $e_2=\{u_1,u_3\}$. 
Since $u_1, u_2, u_{3} \notin L$, we get  
$d(u_i)> \left(\frac{1}{2}-\frac{1}{200} \right)n$.
By Lemma~\ref{lem-W-sub-L}, we have $u_i \notin W$ and 
$d_S(u_i)<  \frac{n}{140}$. Hence
$d_T(u_i)=d(u_i)-d_S(u_i)> \left(\frac{1}{2}-\frac{1}{200}-\frac{1}{140}\right)n.$
By Lemma~\ref{set}, we  get 
    \begin{eqnarray*}
\left |\bigcap_{i=1}^{3}N_T(u_i) \right| &\ge &  \sum_{i=1}^{3}|N_T(u_i)|-2 \left|\bigcup_{i=1}^{3}N_T(u_i) \right| \\
& > & \left (\frac{1}{2}-\frac{1}{200}-\frac{1}{140} \right)n\cdot 3 - 
2 \left(\frac{n}{2}+ 3n^{1/4} \right) \\
 &>& \frac{n}{3}, 
\end{eqnarray*}
where the last inequality holds for $n\ge 166$. 
Consequently, each vertex of the common neighbors of 
$\{u_1,u_2,u_3\}$ leads to at least $3$ new triangular edges,  so $G$ has more than $n$ triangular edges, which is a contradiction.

If $e_1$ and $e_2$ are disjoint, then we denote 
$e_1=\{u_1,u_2\}$ and $e_2=\{u_3,u_4\}$. Similarly, we can see that 
    \begin{eqnarray*}
\left |\bigcap_{i=1}^{4}N_T(u_i) \right| &\ge &  \sum_{i=1}^{4}|N_T(u_i)|-3 \left|\bigcup_{i=1}^{4}N_T(u_i) \right| \\
& > & \left (\frac{1}{2}-\frac{1}{200}-\frac{1}{140} \right)n\cdot 4 - 
3 \left(\frac{n}{2}+ 3n^{1/4} \right) \\
 &>& \frac{n}{4}, 
\end{eqnarray*}
where the last inequality holds for $n\ge 159$. 
In this case, we can also find more than $n$ triangular edges in $G$, a contradiction. Therefore, we conclude that $e(S\setminus L) \le 1$.

Now, by deleting  
at most one vertex from an edge in $G[S\setminus L]$, 
we can obtain 
 a large independent set. Since $|L|\le 9$ 
 by Lemma \ref{Lupper},  
  there exists an independent set
 $I_S\subseteq S\setminus L$ such that
  $
  |I_S|\ge |S\setminus L| - 1 \ge |S|-10$ by Lemma \ref{lem-W-sub-L}. 
 The same argument gives that there  is an independent set $I_T\subseteq T\setminus L$ with
$ |I_T|\ge |T|- 10$.  
 \end{proof}

Let $\mathbf{x}\in \mathbb{R}^n$ be an eigenvector vector corresponding to 
$\lambda (G)$. 
By the Perron--Frobenius theorem, we know that 
$\mathbf{x}$ has all non-negative entries. 
For a vertex $v\in V(G)$, we will write ${x}_v$ for 
 the eigenvector entry of $\mathbf{x}$ corresponding to $v$. 
 Let $z\in V(G)$ be a vertex with the maximum eigenvector entry. 
  Without loss of generality, we may assume by scaling that  ${x}_z=1$ and by symmetry that  $z\in S$.

\begin{lemma} \label{lem-IT}
We have $ \sum\limits_{ v\in I_T} {x}_v >  \frac{n}{2} - 21$.
\end{lemma}

\begin{proof} 
Considering $z$-th entry of 
 the eigenvector equation $A(G)\mathbf{x} = \lambda \mathbf{x}$, we have 
 $$  \frac{n-1}{2} < 
 \lambda (G)= \lambda (G) {x}_z = 
 \sum_{v\in N(z)} {x}_v \le d(z).$$
 Hence $z\notin L$. 
  By Lemma \ref{lem-W-sub-L}, we know that $W\subseteq L$ and 
  $|L|\le 9$. 
From Lemma~\ref{lem-one-edge}, we have 
$d_{S\setminus L}(z) \le 1$ and  
 \[  d_S(z)\le d_{S\setminus L}(z) + |L| \le 10. \] 
 Therefore, we get
 \begin{eqnarray*}
 \lambda (G)&= &  \lambda (G) {x}_z  = 
 \sum_{v\in N_S(z)} {x}_v 
 +\sum_{v\in N_T(z)} {x}_v\\
 &=& \sum_{v\in N_S(z)} {x}_v+  \sum_{v\sim z, v\in I_T} {x}_v 
 +\sum_{v\sim z , v\in T\setminus I_T} {x}_v\\  
 &\le&  10 +\sum_{ v\in I_T} {x}_v+ |T \setminus I_T|\\ 
&\le&  \sum_{ v\in I_T} {x}_v  +20.
\end{eqnarray*}
Recall that $\lambda (G)\ge \lambda (T_{n,2}) > \frac{n-1}{2}$. 
 So 
$ \sum\limits_{ v\in I_T} {x}_v> \frac{n}{2} - 21$, as desired.
\end{proof}

\begin{lemma}\label{Lempty}
We have $L = \varnothing$ and  $e(S) + e(T) \le 1$. 
\end{lemma}

\begin{proof}
By way of contradiction, assume that there is a vertex  $v\in L$, then $d(v)\le (\frac{1}{2}-\frac{1}{200})n$.
We define a graph $G^+$ with the vertex set $V(G)$ and the edge set 
\[  E(G^+) = E(G \setminus \{v\}) \cup \{vw: w\in I_T\}. \] 
 Note that adding a vertex incident with vertices in $I_T$ does not create any triangular edges   
since $I_T$ is an independent set.  
By Lemma \ref{lem-IT},  we have 
\begin{eqnarray*}
\lambda (G^+) - \lambda (G) &\geq &
 \frac{\mathbf{x}^{\top} \left(A(G^+) - A(G)\right) \mathbf{x}}{\mathbf{x}^{\top} \mathbf{x}} \\
 & =& \frac{2 {x}_v}{\mathbf{x}^{\top}\mathbf{x}} 
 \biggl( \sum_{w\in I_T} {x}_w - \sum_{u\in N_G(v)} {x}_u\biggr) \\
& >& \frac{2 {x}_v}{\mathbf{x}^{\top} \mathbf{x}} 
\left( \frac{n}{2} - 21 - 
\left(\frac{1}{2}-\frac{1}{200} \right)n \right)\\
&=& \frac{2 {x}_v}{\mathbf{x}^{\top} \mathbf{x}} \left( \frac{n}{200}-21\right)> 0,
\end{eqnarray*}
where the last inequality holds for $n> 4200$. This contradicts 
with the maximality of the spectral radius  of $G$, so $L$ must be empty.

By Lemma \ref{lem-one-edge}, 
we  get 
$e(S) \le 1$ and $e(T) \le 1$. 
Since $L=\varnothing$, then for every vertex $v\in S$, we have 
$d(v) > (\frac{1}{2} - \frac{1}{200})n $ and 
$d_S(v) \le 1$. So $d_T(v) \ge \lfloor (\frac{1}{2} - \frac{1}{200} )n \rfloor$. The corresponding degree condition also holds for 
each vertex of $T$. We next show $e(S) + e(T)\le 1$. 
Assume otherwise, so that $e(S)=1 $ and $ e(T)=1$. 
Then we denote $e_1=\{v_1,v_2\}\in E(G[S])$. 
Observe that for $n\ge 137$, we have 
\[  |N_T(v_1) \cap N_T(v_2)| >  2 \left\lfloor \left(\frac{1}{2} - \frac{1}{200} \right)n \right\rfloor - \left(\frac{n}{2} + 3n^{1/4} \right) > \frac{2n}{5}.\]
Each vertex of the common neighbors of $v_1,v_2$ in $T$ can yield 
two triangular edges. There are more than $\frac{4}{5}n$ triangular edges 
between $\{v_1,v_2\}$ and $N_T(v_1) \cap N_T(v_2)$.  
Similarly, the edge in $G[T]$ can lead to at least $\frac{4n}{5} -4$ 
new triangular edges,  so $G$ has 
more than $\frac{7}{5}n$ triangular edges. 
This is a contradiction. Therefore, we have 
$e(S) + e(T)\le 1$, as required. 
\end{proof}

The most general result is the following structure theorem, which 
asserts that any graph with larger spectral radius than $T_{n,2}$ 
and few triangular edges can be approximated by 
an almost-balanced complete bipartite graph. 
Just like in the classical stability method, 
once we have proved that the extremal graph is quite close to the 
conjectured graph, we can show further that it
must be exactly the conjectured graph.

\begin{theorem}\label{STlambdarefine}
If $G$ is a graph of order $n$ with 
at most $n+1$ triangular edges, and $G$ has the maximum spectral radius, then 
$e(G)\ge \lfloor {n^2}/{4} \rfloor -3 $. 
Moreover, 
 there exists a vertex partition $V(G)=S\cup T$ such that  
$e(S,T)\ge \lfloor n^2/4 \rfloor -4$ and 
$\lceil {n}/{2} \rceil -2
\le |S|,  |T|\le \lfloor {n}/{2} \rfloor +2$. 
\end{theorem}

\begin{proof}
 From Lemma \ref{Lempty}, 
  we have $e(S) + e(T) \leq 1$.
  Since any triangle contains an  edge of  $E(S)\cup E(T)$, 
  the number of triangles in $G$ is bounded above by $\frac{n}{2} +3n^{1/4}$. 
       By Lemma \ref{thm-BN-CFTZ-NZ},
       we have $$e(G) \geq \lambda^2-\frac{6t}{n-1}> 
       \left\lfloor \frac{n^2}{4}  \right \rfloor-4.$$
Then 
\[ e(S,T) = e(G) - e(S) - e(T) > \frac{n^2}{4} -5. \]
By symmetry, we may assume that $|S| \le |T|$. 
       Suppose on the contrary  that $|S|\le \lceil \frac{n}{2} \rceil -3$.  Then $|T|=n-|S|\ge \lfloor \frac{n}{2} \rfloor +3$. 
       If $n$ is even, then it follows that 
       $e(S,T)\le  |S||T|\le  \left (\frac{n}{2}-3 \right) \left(\frac{n}{2}+3 \right)= \frac{n^2}{4}-9$, 
        which contradicts with $e(S,T)\ge {n^2}/{4} - 4$. 
        If $n$ is odd, then $e(S,T)\le (\frac{n+1}{2}-3)(\frac{n-1}{2} +3) =\frac{n^2-1}{4} -6$, a contradiction. 
        Thus, we have
        \[   \left\lceil \frac{n}{2} \right\rceil -2
\le |S|, |T|\le  \left\lfloor \frac{n}{2} \right\rfloor +2. \] 
This completes the proof. 
\end{proof}

Now, we are ready to present the proof of Theorem \ref{thm-main}. 

\begin{proof}[{\bf Proof of Theorem \ref{thm-main}}] 
Let $G$ be a graph on $n\ge 5432$ vertices with $\lambda (G)\ge \lambda (T_{n,2})$ and $G\neq T_{n,2}$. 
Suppose on the contrary that $G$ has at most $2 \lfloor n/2 \rfloor -2$ triangular edges. 
Furthermore, we also choose $G$ as a graph with the maximum spectral radius. In what follows, we will deduce a contradiction.  

\medskip 
First of all, we know from Theorem \ref{thmNZ2021} that $G$ contains at least $\lfloor n/2\rfloor -1$ triangles\footnote{
We use Theorem \ref{thmNZ2021} in order to avoid the complicated computations.}. 
By Theorem \ref{STlambdarefine}, 
$G$ is almost complete bipartite, 
and we have $n/2 -2\le |S|,|T| \le n/2 +2$.  
If $e(S) + e(T)=0$, 
then $G$ is a bipartite graph with color classes $S$ and $T$. 
So we have $\lambda (G)\le \sqrt{|S||T|} 
\le \sqrt{\lfloor n^2/4\rfloor}$ 
since $|S| + |T|=n$. 
On the other hand, our assumption gives 
$\lambda (G)\ge \lambda (T_{n,2}) = \sqrt{\lfloor n^2/4\rfloor}$.
Therefore, it follows that $G=T_{n,2}$, a contradiction. 
By Lemma \ref{Lempty}, 
we now assume that $e(S) + e(T)=1$. 
Next, we divide the proof into two cases.

\medskip 
{\bf Case 1.} Assume that $n$ is even.

{\bf Subcase 1.1.}  $|S|= \frac{n}{2} -2$ and $|T|=\frac{n}{2} +2$. 
 If $e(S)=1$, then $G$ is a subgraph of $K_{\frac{n}{2} -2, \frac{n}{2} +2}^+$.  
Similarly, we get that 
$\lambda (K_{\frac{n}{2} -2, \frac{n}{2} +2}^+) $ is the largest root of 
\[  g_1(x)=x^3 -x^2 + 4 x- (n^2x)/4  + n^2/4 -n -8. \]
We can check that $g_1(\frac{n}{2}) = n-8 >0$ and 
$g_1'(x) \ge 0$ for every $x\ge \frac{n}{2}$. It follows that 
$ \lambda (K_{\frac{n}{2} -2, \frac{n}{2} +2}^+) < \frac{n}{2}$. 
If $e(T)=1$, then $G$ is a subgraph of $K_{\frac{n}{2} +2, \frac{n}{2} -2}^+$. By computation, we obtain that $\lambda (K_{\frac{n}{2} +2, \frac{n}{2} -2}^+) $ is the largest root of 
\[   g_2(x)= x^3 -x^2 + 4x - (n^2 x)/4 + n^2/4 -n.   \]
It is easy to verify that $g_2(\frac{n}{2}) = n >0$ and $g_2'(x) \ge 0$ 
for $x\ge \frac{n}{2}$. Thus, we have $\lambda (K_{\frac{n}{2} +2, \frac{n}{2} -2}^+) < \frac{n}{2} = \lambda (T_{n,2})$, a contradiction. 
Apart from the direct computation, 
there is another way to see that 
$\lambda (K_{\frac{n}{2} +2, \frac{n}{2} -2}^+) < \lambda (T_{n,2})$. 
Suppose in contrast that $\lambda (K_{\frac{n}{2} +2, \frac{n}{2} -2}^+) \ge \lambda (T_{n,2})$. 
Then Theorem \ref{thmNZ2021} implies that $K_{\frac{n}{2} +2, \frac{n}{2} -2}^+$ contains at least 
$\frac{n}{2} -1$ triangles, which is a contradiction immediately. 

In this subcase, we conclude that 
 either $\lambda (G) \le \lambda (K_{\frac{n}{2} +2, \frac{n}{2} -2}^+) < \lambda (T_{n,2}) $ or 
 $\lambda (G)\le  \lambda (K_{\frac{n}{2} -2, \frac{n}{2} +2}^+) < \lambda (T_{n,2}) $, which contradicts with the assumption. 

{\bf Subcase 1.2.}  $|S|= \frac{n}{2} -1$ and $|T|=\frac{n}{2} +1$.  
If $e(S)=1$, then $G$ is a subgraph of $K_{\frac{n}{2} -1, \frac{n}{2} +1}^+$. Since $G$ has at most $n-2$ triangular edges, 
and $K_{\frac{n}{2} -1, \frac{n}{2} +1}^+$ has $2|T|+1=n+3$ triangular edges. 
Therefore, we must destroy at least $5$ triangular edges from 
$K_{\frac{n}{2} -1, \frac{n}{2} +1}^+$ to obtain the graph $G$. 
Consequently, the deleted triangular edges are incident to at least $3$ vertices of $T$. 
Then $G$ has at most $|T| -3 = \frac{n}{2}-2$ triangles, a contradiction.  
If $e(T)=1$, then 
$G$ is a subgraph of $K_{\frac{n}{2} +1, \frac{n}{2} -1}^+$. 
Observe that $K_{\frac{n}{2} +1, \frac{n}{2} -1}^+$ has $n-1$ triangular edges. We must delete at least one triangular edge of $K_{\frac{n}{2} +1, \frac{n}{2} -1}^+$ to obtain $G$. 
It follows that $G$ has at most $\frac{n}{2} -2$ 
 triangles, a contradiction.

{\bf Subcase 1.3.}  $|S|= \frac{n}{2} $ and $|T|=\frac{n}{2} $.  
In this situation, 
we may assume by the symmetry that $e(S)=1$. 
Then $G$ is a subgraph of $K_{\frac{n}{2} , \frac{n}{2}}^+$. 
Recall that $G$ has at most $n-2$ triangular edges, and 
$K_{\frac{n}{2} , \frac{n}{2}}^+$ has exactly $n+1$ triangular edges. 
To obtain the graph $G$,  we need to destroy at least $3$ triangular edges 
 from $K_{\frac{n}{2} , \frac{n}{2}}^+$.  
 Consequently, $G$ has at most $\frac{n}{2} -2$ triangles, 
 which is a contradiction.

\medskip 
{\bf Case 2.} Suppose that $n$ is odd. In this case, by assumption, 
we know that 
$G$ contains at least $\frac{n-3}{2}$ triangles 
and $G$ has at most $n-3$ triangular edges.  

{\bf Subcase 2.1.}  $|S|= \frac{n-3}{2} $ and $|T|=\frac{n+3}{2}$. 
If $e(S)=1$, then $G$ is a subgraph of $K_{\frac{n-3}{2} , \frac{n+3}{2} }^+$. 
Notice that $K_{\frac{n-3}{2} , \frac{n+3}{2} }^+$ has exactly 
$n+4$ triangular edges. 
To obtain the graph $G$, we need to destroy 
at least $7$ triangular edges. 
Then we need to delete some triangular edges that 
are incident to at least $4$ vertices of $T$,  so 
$G$ has at most $|T|-4=\frac{n-5}{2} $ triangles, 
a contradiction. 
If $e(T)=1$, then $G$ is a subgraph of $K_{\frac{n+3}{2} , \frac{n-3}{2} }^+$. By computation, we obtain that 
$\lambda (K_{\frac{n+3}{2} , \frac{n-3}{2} }^+)$ is the largest root of 
\[ g_3(x)=x^3 - x^2 + (9 x)/4 - (x n^2)/4 + n^2/4 - n + 3/4.  \]
It is easy to check that $g_3(\frac{1}{2}\sqrt{n^2-1})= 
1 - n + \sqrt{ n^2-1}>0$. 
Moreover, we have $g_3'(x)=3x^2 -2x + 9/4 - n^2/4$. 
We can verify that $g_3'(x)>0$ for any $x> \frac{1}{2}\sqrt{n^2-1}$, 
which yields $ 
\lambda (G)\le \lambda (K_{\frac{n+3}{2} , \frac{n-3}{2} }^+) <  \frac{1}{2}\sqrt{n^2-1} = \lambda (T_{n,2})$, a contradiction. 

{\bf Subcase 2.2.}  $|S|= \frac{n-1}{2} $ and $|T|=\frac{n+1}{2}$. 
If $e(S)=1$, then $G$ is a subgraph of $K_{\frac{n-1}{2} , \frac{n+1}{2} }^+$. Since $K_{\frac{n-1}{2} , \frac{n+1}{2} }^+$ has exactly 
$n+2$ triangular edges, we must destroy at least 
$5$ triangular edges to obtain the graph $G$. 
So the deleted triangular edges are incident to at least $3$ vertices of $T$, and $G$ contains at most $|T|-3=\frac{n-5}{2}$ triangles, 
a contradiction. 
If $e(T)=1$, then $G$ is a subgraph of $K_{\frac{n+1}{2} , \frac{n-1}{2} }^+$. As  $K_{\frac{n+1}{2} , \frac{n-1}{2} }^+$ has 
$n$ triangular edges, we need to destroy at least 
$3$ triangular edges to produce  $G$. 
In this process, 
at least two triangles of $K_{\frac{n+1}{2} , \frac{n-1}{2} }^+$ 
are removed, so $G$ has at most $|S|-2=\frac{n-5}{2}$ triangles, 
which is a contradiction.  
\end{proof}

\noindent 
{\bf Remark.} 
In the above proof,  we can determine 
the extremal graphs $G$ in the sense that
 $\lambda (G)\ge \lambda (T_{n,2})$, $G\neq T_{n,2}$ and $G$ has exactly $2\lfloor \frac{n}{2}\rfloor -1$ triangular edges. 
Indeed, we next give the sketch without details. 

\medskip  
In Subcase 1.1, it was proved that 
$\lambda (G) < \lambda (T_{n,2})$, a contradiction.

In Subcase 1.2, 
as we know, $G$ has exactly $n-1$ triangular edges. 
If $e(S)=1$,   
then $G$ is obtained  
from $K_{\frac{n}{2}-1,\frac{n}{2}+1}^+$ 
by deleting at least two triangular edges that incident to 
two vertices of $T$. 
In this deletion, we destroy 
four triangular edges of $K_{\frac{n}{2}-1,\frac{n}{2}+1}^+$. 
More precisely, let $\{u,v\}$ be the unique edge of $G[S]$. Then we can delete two triangular edges 
that intersect in $u$, or delete two disjoint triangular edges incident to $u$ and $v$, respectively. 
In each case, we can compute that 
the resulting graphs have spectral radius less than $\lambda (T_{n,2})$.  
If $e(T)=1$, then $G$ is a subgraph of $K_{\frac{n}{2} +1, \frac{n}{2} -1}^+$. 
Note that we cannot delete any triangular edges to obtain $G$.  
Moreover, we can verify that 
the deletion of a non-triangular edge
leads to a graph with spectral radius less than $\lambda (T_{n,2})$. 
So we have $G=K_{\frac{n}{2} +1, \frac{n}{2} -1}^+$. 
In addition, Lemma \ref{lem-pm-1} gives 
 $\lambda (K_{\frac{n}{2} +1, \frac{n}{2} -1}^+) > \lambda (T_{n,2})$. 
Thus $K_{\frac{n}{2} +1, \frac{n}{2} -1}^+$ is 
one of the extremal graphs.  

In Subcase 1.3, 
$G$ is obtained from $K_{\frac{n}{2} , \frac{n}{2}}^+$ by 
deleting at least one triangular edge. 
So $G$ is a subgraph of $K_{\frac{n}{2} ,  \frac{n}{2} }^{+|}$. 
 By calculation,  deleting any edge from $K_{\frac{n}{2} ,  \frac{n}{2} }^{+|}$ yields a graph with spectral radius 
 less than $\lambda (T_{n,2})$. 
 Then we must have $G=K_{\frac{n}{2} ,  \frac{n}{2} }^{+|}$. 
 From Lemma \ref{lem-equal-pm}, 
we get 
$\lambda (K_{\frac{n}{2} ,  \frac{n}{2} }^{+|}) > \lambda (T_{n,2})$. So $K_{\frac{n}{2} ,  \frac{n}{2} }^{+\,|}$ is the second  extremal graph. 

\medskip 
In Subcase 2.1, $G$ has exactly $n-2$ triangular edges. 
If $e(S)=1$, then $G$ is 
a subgraph of $K_{\frac{n-3}{2} , \frac{n+3}{2} }^+$ 
by deleting at least three triangular edges incident to 
three vertices of $T$. 
For example, let $\{u,v\}$ be the unique edge of $G[S]$. 
 We can delete three triangular edges that intersect in $u$, 
or we delete two triangular edges incident to $u$, 
and one triangular edge incident to $v$. 
In the two cases, 
the resulting subgraphs have spectral radius less than 
$\lambda (T_{n,2})$. 
If $e(T)=1$, then $G$ is a subgraph of 
$K_{\frac{n+3}{2} , \frac{n-3}{2} }^+$. In the previous proof, 
we have shown that $ \lambda (G)< \lambda (T_{n,2})$, a contradiction.

In Subcase 2.2, 
if $e(S)=1$, then $G$ is a subgraph of $K_{\frac{n-1}{2} , \frac{n+1}{2} }^+$. 
Note that $K_{\frac{n-1}{2} , \frac{n+1}{2} }^+$ contains 
$n+2$ triangular edges. We need to delete at least 
two triangular edges incident to two vertices of $T$. 
Let $\{u,v\}$ be the unique edge 
of $G[S]$. Then 
$G$ can be obtained by deleting two triangular edges 
that intersect in $u$, or 
deleting two disjoint triangular edges incident to $u$ and $v$, 
respectively. 
 In both cases, 
 we can check that the resulting graphs have spectral radius less than $\lambda (T_{n,2})$.  
If $e(T)=1$, then $G$ is a subgraph of 
$K_{\frac{n+1}{2} , \frac{n-1}{2} }^{+\,|}$. 
We can calculate that any proper subgraph has spectral radius less than $\lambda (T_{n,2})$. 
Moreover, 
 Lemma \ref{odd-n-pm-one} tells us that 
$\lambda (K_{\frac{n+1}{2} , \frac{n-1}{2} }^{+\,|}) 
> \lambda (T_{n,2})$, so $K_{\frac{n+1}{2} , \frac{n-1}{2} }^{+\,|}$  is the third extremal graph.

\section{Proof of Theorem \ref{thm-main3}}

\label{sec6}

Using a similar argument, 
we can prove Theorem \ref{thm-main3}. 

\begin{proof}[{\bf Proof of Theorem \ref{thm-main3}}]
Let  $G$ be an $n$-vertex graph with 
$\lambda (G) \ge \lambda (K_{\lceil \frac{n}{2}\rceil , \lfloor \frac{n}{2}\rfloor}^+)$ and 
$G$ has at most $2 \lfloor n/2 \rfloor +1$ triangular edges.  
We shall show that $G=K_{\lceil \frac{n}{2}\rceil , \lfloor \frac{n}{2}\rfloor}^+$.  
First of all, 
we know from Theorem \ref{STlambdarefine} that
 $G$ is an almost balanced complete bipartite graph. 
 More precisely, we have $e(G)\ge \lfloor {n^2}/{4} \rfloor -3 $, and $G$ admits a partition $V(G)=S\cup T$ such that  
$e(S,T)\ge \lfloor n^2/4 \rfloor -4$ and 
$n/2 -2\le |S|,|T| \le n/2 +2$.   
 If $e(S) + e(T)=0$,  
then $G$ is a bipartite graph with color classes $S$ and $T$. 
Consequently, we get
 $\lambda (G)
\le \lambda (T_{n,2}) < \lambda (K_{\lceil \frac{n}{2}\rceil , \lfloor \frac{n}{2}\rfloor}^+)$, which contradicts with 
the assumption. 
By Lemma \ref{Lempty}, 
we  have $e(S) + e(T)=1$. 
In what follows, we divide the proof into two cases.

\medskip 
{\bf Case 1.} Assume that $n$ is even. 

\medskip 
{\bf Subcase 1.1.}  $|S|= \frac{n}{2} -2$ and $|T|=\frac{n}{2} +2$. 
If $e(T)=1$, then $G$ is a subgraph of $K_{\frac{n}{2} +2, \frac{n}{2} -2}^+$. In the proof of 
Theorem \ref{thm-main} for Subcase 1.1,  
we have shown that 
$\lambda (K_{\frac{n}{2} +2, \frac{n}{2} -2}^+) < \frac{n}{2}$, a contradiction. 
 If $e(S)=1$, then $G$ is a subgraph of $K_{\frac{n}{2} -2, \frac{n}{2} +2}^+$.  We also showed that 
 $\lambda (G)\le  \lambda (K_{\frac{n}{2} -2, \frac{n}{2} +2}^+) < \frac{n}{2}$, which contradicts with the assumption. 

\medskip 
{\bf Subcase 1.2.}  $|S|= \frac{n}{2} -1$ and $|T|=\frac{n}{2} +1$.  
If $e(S)=1$, then $G$ is a subgraph of $K_{\frac{n}{2} -1, \frac{n}{2} +1}^+$. Similarly, we can show that 
\[ \lambda (K_{\frac{n}{2} -1, \frac{n}{2} +1}^+) < 
\lambda (K_{\frac{n}{2} , \frac{n}{2} }^+). \]  
Indeed, $\lambda (K_{\frac{n}{2} -1, \frac{n}{2} +1}^+)$ 
is the largest root of 
\[  h_1(x)= -3 - n + n^2/4 + x - (n^2 x)/4 - x^2 + x^3. \]
Recall in Lemma \ref{lem-21} that 
$\lambda (K_{\frac{n}{2} , \frac{n}{2} }^+)$ is the largest root of 
\[ f(x) =-n + n^2/4 - (n^2 x)/4 - x^2 + x^3.  \]
Observe that $h_1(x) - f(x) = x-3 >0$ for every $x>3$. 
Then we have $h_1(x) > f(x) \ge 0$ for any 
$x\ge \lambda (K_{\frac{n}{2} , \frac{n}{2} }^+)$, 
which implies $\lambda (K_{\frac{n}{2} -1, \frac{n}{2} +1}^+) < 
\lambda (K_{\frac{n}{2} , \frac{n}{2} }^+)$, as needed.

If $e(T)=1$, then 
$G$ is a subgraph of $K_{\frac{n}{2} +1, \frac{n}{2} -1}^+$. 
We can prove that 
\[  \lambda (K_{\frac{n}{2} +1, \frac{n}{2} -1}^+) < \lambda (K_{\frac{n}{2}, \frac{n}{2} }^+). \]
Indeed, since $\lambda (K_{\frac{n}{2} +1, \frac{n}{2} -1}^+)$ 
 is the largest root of 
 \[ h_2(x) = 1 - n + n^2/4 + x - (n^2 x)/4 - x^2 + x^3,  \]
 and $h_2(x) > f(x)$ for any 
 $x>0$, which yields $\lambda (K_{\frac{n}{2} +1, \frac{n}{2} -1}^+) < \lambda (K_{\frac{n}{2}, \frac{n}{2} }^+)$.

\medskip  
{\bf Subcase 1.3.}  $|S|= \frac{n}{2} $ and $|T|=\frac{n}{2} $.  
By the symmetry, 
we may assume that $e(S)=1$. 
Then $G$ is a subgraph of $K_{\frac{n}{2} , \frac{n}{2}}^+$. 
Since $\lambda (G) \ge \lambda (K_{\frac{n}{2}, 
\frac{n}{2}}^+)$, we get $G= K_{\frac{n}{2} , \frac{n}{2}}^+$, 
which is the desired extremal graph.

\medskip 
{\bf Case 2.} Suppose that $n$ is odd. 

\medskip 
{\bf Subcase 2.1.}  $|S|= \frac{n-3}{2} $ and $|T|=\frac{n+3}{2}$. 
If $e(S)=1$, then $G$ is a subgraph of $K_{\frac{n-3}{2} , \frac{n+3}{2} }^+$. 
By calculation, we obtain that  $\lambda (K_{\frac{n-3}{2} , \frac{n+3}{2} }^+)$ is the largest root of 
\[ h_3(x) =-(21/4) - n + n^2/4 + (9 x)/4 - (n^2 x)/4 - x^2 + x^3. \]
By Lemma  \ref{lem-21}, 
we know that $\lambda(K_{\frac{n+1}{2}, \frac{n-1}{2}}^+)$
is the largest root of
\[ g(x)=3/4 - n + n^2/4 + x/4 - (n^2 x)/4 - x^2 + x^3. \]
Since $h_3(x)- g(x) = 2x-6$, we get 
$h_3(x) > g(x)$ for any $x>3$, so it follows that 
$\lambda (K_{\frac{n-3}{2} , \frac{n+3}{2} }^+) < 
\lambda(K_{\frac{n+1}{2}, \frac{n-1}{2}}^+)$, a contradiction.

If $e(T)=1$, then $G$ is a subgraph of $K_{\frac{n+3}{2} , \frac{n-3}{2} }^+$. By computation, we obtain that 
$\lambda (K_{\frac{n+3}{2} , \frac{n-3}{2} }^+)$ is the largest root of 
\[ h_4(x)=3/4 - n + n^2/4 + (9 x)/4 - (n^2 x)/4 - x^2 + x^3.  \]
It is easy to check that $h_4(x) > g(x)$ for any $x>0$. 
So we have $\lambda (K_{\frac{n+3}{2} , \frac{n-3}{2} }^+) < 
\lambda (K_{\frac{n+1}{2} , \frac{n-1}{2} }^+)$, 
which contradicts with the assumption on $G$. 

\medskip 
{\bf Subcase 2.2.}  $|S|= \frac{n-1}{2} $ and $|T|=\frac{n+1}{2}$. 
If $e(S)=1$, then $G$ is a subgraph of $K_{\frac{n-1}{2} , \frac{n+1}{2} }^+$. Since $G$ has at most $n$ triangular edges and $K_{\frac{n-1}{2} , \frac{n+1}{2} }^+$ has exactly $n+2$ triangular edges, 
we must destroy at least 
two triangular edges to obtain the subgraph $G$. 
Let $K_{\frac{n-1}{2} , \frac{n+1}{2} }^{+\,|}$ 
be the graph obtained from 
$K_{\frac{n-1}{2} , \frac{n+1}{2} }^+$ by deleting an edge between $S$ and $T$ such that this edge is incident to 
the unique edge of $G[S]$. 
Furthermore, it follows that 
$G$ is a subgraph of $K_{\frac{n-1}{2} , \frac{n+1}{2} }^{+\,|}$.  
In this case, we can show that 
\[  \lambda (K_{\frac{n-1}{2} , \frac{n+1}{2} }^{+\, |}) < \frac{n}{2} 
< \lambda (K_{\frac{n+1}{2} , \frac{n-1}{2} }^{+}) .  \]
Indeed, since the spectral radius of $K_{\frac{n-1}{2} , \frac{n+1}{2} }^{+\, |}$ 
is the largest root of 
\[ h_5(x) = -(x/2) - 2 n x + (n^2 x)/2 + x^2 - n x^2 + x^3/4 - (n^2 x^3)/4 + x^5,  \]
and $h_5(\frac{n}{2})=  (-8 n - 24 n^2 + n^3)/32 >0$
for $n\ge 25$. Moreover, we can check that 
$h_5'(x) >0$ for every $x\ge \frac{n}{2}$. So it yields 
$\lambda (K_{\frac{n-1}{2} , \frac{n+1}{2} }^{+\, |}) < \frac{n}{2} < 
 \lambda (K_{\frac{n+1}{2} , \frac{n-1}{2} }^{+})$ by Lemma \ref{lem-21}.  
This is a contradiction. 

If $e(T)=1$, then $G$ is a subgraph of $K_{\frac{n+1}{2} , \frac{n-1}{2} }^+$. The assumption asserts that $\lambda (G) 
\ge \lambda (K_{\frac{n+1}{2} , \frac{n-1}{2} }^+)$, 
so we get $G=K_{\frac{n+1}{2} , \frac{n-1}{2} }^+$, 
which is the expected extremal graph. 
\end{proof}

\section{Concluding remarks}

\label{sec7}

In this paper, 
we have bounded the minimum number of triangular edges 
of a graph in terms of the spectral radius, 
and we have established a spectral Erd\H{o}s--Faudree--Rousseau theorem. 
The main ideas in our proof 
attribute to the supersaturation-stability (Theorem \ref{thm-far-bipartite}) and some additional spectral techniques. 
We believe that this method may have the potential to be applied to a wider range of spectral extremal graph problems.

\subsection{Supersaturation-stability via spectral radius}
We stated in Subsection \ref{sec4-3} that 
Theorem \ref{thm-far-bipartite} can deduce a conjecture 
of Erd\H{o}s involving the booksize of a graph. 
It is worth mentioning that an interesting spectral problem 
of Zhai and Lin \cite[Problem 1.2]{ZL2022jgt} asserts that 
every $n$-vertex graph $G$ with 
$\lambda (G) > \lambda (T_{n,2})$ has booksize greater than $n/6$ as well. 
To solve this problem, 
it is sufficient to show a spectral version of 
Theorem \ref{thm-far-bipartite}. 
For the sake of formality, we propose the following conjecture. 

\begin{conjecture} \label{conj-spec-far}
If $G$ is $t$-far from being bipartite, then 
\[  t(G) \ge \frac{n}{6}\Bigl( \lambda(G) + t - \lambda (T_{n,2}) \Bigr). \]
\end{conjecture}

\noindent 
{\bf Remark.} 
Apart from being interesting on its own, we can see that, somewhat surprisingly, 
Conjecture \ref{conj-spec-far} in fact implies the aforementioned problem of Zhai and Lin \cite[Problem 1.2]{ZL2022jgt}. 
Indeed, suppose that $G$ is an $n$-vertex graph with $\lambda (G)> \lambda (T_{n,2})$, and 
$G$ contains exactly $t$ triangles.  
Then assuming Conjecture \ref{conj-spec-far}, we know that 
$G$ is not ${6t}/{n}$-far from being bipartite. 
So we can remove less than ${6t}/{n}$ edges from $G$ to destroy all $t$ triangles. 
Thus,  one of these edges is contained in more than $n/6$ triangles, as expected.

\medskip 
We have also proved in Subsection \ref{sec4-3}  that the spectral extremal result \cite[Theorem 2]{CFTZ20} for the friendship graph $F_k$ holds for every $n\ge (21k)^4$ by applying  Lemma \ref{lem-Fk}. We point out here that the constant factor can be slightly improved by a  result in \cite[Theorem 4]{Zhu2023}, which shows that for $n\ge 4k^3$,  every $n$-vertex $F_k$-free graph  contains less than $k^2n$ triangles. 
This leads to an improvement on the coefficients of Lemma \ref{lem-Fk} under the constraint $n\ge 4k^3$. 
 However, it seems difficult to improve the exponent of $k$. 
In the case of Tur\'{a}n number, 
 Erd\H{o}s,  F\"{u}redi, Gould and Gunderson \cite{Erdos95}  
 proved that for $k\ge 1$ and $n\ge 50k^2$, we have 
 \[ \mathrm{ex}(n, F_k)= \left\lfloor \frac {n^2}{4}\right \rfloor+ \left\{
  \begin{array}{ll}
   k^2-k, \quad~~  \mbox{if $k$ is odd;} \\
    k^2-\frac32 k, \quad \mbox{if $k$ is even}.
  \end{array}
\right. \] 
 Furthermore, they conjectured \cite[page 90]{Erdos95} 
 that the above result on $\mathrm{ex}(n,F_k)$ still holds for every $n\ge 4k$, rather than $n\ge 50k^2$. 
 This conjecture remains {\it unresolved}. 
 In the case of spectral radius,  
we may ask further that whether the spectral extremal result for $F_k$ still holds for every $n\ge Ck$ with an absolute constant $C>0$. 
We mention that finding (linear) sharp bounds on the order of graphs is also regarded as an interesting problem in extremal graph theory, we refer the readers to \cite{FG2015cpc, LN2021outplanar, ZL2022jgt} for some related results.

\subsection{Counting triangular edges} 
A well-known result of Nosal \cite{Nosal1970} 
(see, e.g., \cite{Niki2002cpc,Ning2017-ars}) 
asserts that if $G$ is a  graph with
$m$ edges and $\lambda (G) > \sqrt{m}$,
then it contains a triangle.
In 2023, Ning and Zhai \cite{NZ2021}
proved a counting result, which asserts that if
$  \lambda (G)\ge \sqrt{m}$, 
then $t(G)\ge \lfloor \frac{1}{2}(\sqrt{m}-1) \rfloor$, unless $G$ is a complete bipartite graph. 
Inspired by this result, 
we propose the following problem.

\begin{conjecture} \label{thm-main2}
If $G$ is a graph with $m$ edges and 
\[  \lambda (G) \ge \sqrt{m} , \]
then $G$ has at least $\sqrt{m}$ triangular edges, unless 
$G$ is a complete bipartite graph. 
\end{conjecture}

\medskip 
In what follows, 
we shall conclude some problems 
concerning the minimum number of edges that occur in cliques or odd cycles. 
Motivated by the study on the minimum number of 
 triangular edges
 among graphs with $n$ vertices and $m\ge n^2/4 +q $ edges, 
 we may ask the following conjecture, 
 which provides a spectral version of 
 Theorem \ref{thm-GL}. 
 
 \begin{conjecture}  
For any graph $G$ on  $n$ vertices, there exists 
an $n$-vertex graph $H=G(a,b,c)$  for some integers $a,b,c$ such that $\lambda (H)\ge \lambda(G)$  and $|\mathbf{NT}(H)| \ge |\mathbf{NT}(G)|$. 
\end{conjecture}

 \subsection{Counting edges in cliques or odd cycles}
 
Recall that $T_{n,r}$ is the $r$-partite Tur\'{a}n graph on $n$ vertices. 
The famous Tur\'{a}n theorem  \cite[p. 294]{Bollobas78} 
states that an $n$-vertex  graph $G$ with 
   $ e(G)\ge e(T_{n,r})$ has a copy of $K_{r+1}$, unless  $G=T_{n,r}$. Correspondingly, 
   Nikiforov \cite{Niki2007laa2} showed that if 
   $\lambda (G) \ge \lambda (T_{n,r})$, then 
   $G$ contains a copy of $K_{r+1}$, unless $G=T_{n,r}$. 
   So it is natural to consider the following extension 
   by minimizing the number of edges that occur in $K_{r+1}$. 

\begin{problem} \label{prob-clique}
Suppose that $r\ge 3$ and 
$G$ is an $n$-vertex graph with 
$\lambda (G) > \lambda (T_{n,r})$. 
What is the smallest number of edges  of $G$ that 
are contained in  $K_{r+1}$?  
\end{problem}

\noindent 
{\bf Remark.} 
Inspired by Conjecture \ref{conj-FM} and Theorem 
\ref{thm-GL}, 
we believe intuitively that the spectral extremal graphs 
in Problem \ref{prob-clique} 
  are possibly analogues of graphs of the form $G(a,b,c)$, i.e., 
 they are perhaps constructed from 
 a complete $r$-partite graph of order $n$ by 
adding an almost complete graph to one of the vertex parts.

Apart from the number of triangular edges, 
Erd\H{o}s,  Faudree and Rousseau \cite{EFR1992} 
also considered the analogous problems 
for longer odd cycles in a graph of order $n$ with 
more than $\lfloor n^2/4 \rfloor$ edges. 
They proved that for any $k\ge 2$, every 
graph on $n$ vertices with $\lfloor n^2/4\rfloor +1$ 
edges contains at least $\frac{11}{144}n^2 - O(n)$ 
edges that are contained in an odd cycle $C_{2k+1}$.  
It turns out that the case $k\ge 2$ is quite different from 
the triangle case.  
Furthermore, 
Erd\H{o}s,  Faudree and Rousseau \cite{EFR1992} made 
a stronger conjecture, 
which asserts that all such graphs 
contain at least $\frac{2}{9}n^2 - O(n)$ edges that occur in $C_{2k+1}$. 
We remark that adding an extra edge 
to the complete balanced bipartite graph is not optimal.

In 2017, F\"{u}redi and Maleki \cite{FM2017}  
disproved this conjecture for $k=2$, and  
constructed $n$-vertex graphs with $\lfloor n^2/4 \rfloor +1$ 
edges and with only $ \frac{2+\sqrt{2}}{16} n^2 + O(n) \approx 
0.213 n^2$ edges in $C_5$. 
In 2019, Grzesik, Hu and Volec 
\cite{GHV2019} obtained 
asymptotically sharp bounds for the smallest possible
number of edges in $C_{2k+1}$. Using Razborov's  flag algebras method, they proved that if 
$G$ is an $n$-vertex graph with $\lfloor n^2/4 \rfloor +1$ edges, 
then it contains at least $\frac{2+\sqrt{2}}{16} n^2 - O(n^{15/8})$ 
edges that occur in $C_5$, 
and for $k\ge 3$, it contains at least $\frac{2}{9}n^2 - O(n)$ 
edges in $C_{2k+1}$. 
Motivated by these results, 
we may propose the following spectral problem. 

\begin{problem}
Let $G$ be a graph of order $n$ with $\lambda (G) > \lambda (T_{n,2})$. 
For each $k\ge 2$, what is the smallest number of edges of $G$ that 
occur in $C_{2k+1}$?  
\end{problem}

\section*{Acknowledgements}  
The authors would like to thank Xiaocong He and Loujun Yu 
for carefully reading an early manuscript of this paper. 
 The authors also show their great gratitude to anonymous referees for valuable suggestions, which considerably improve the presentation of the paper. 
Yongtao Li was supported by the Postdoctoral Fellowship Program of CPSF (No. GZC20233196), 
 Lihua Feng was supported by the National Natural Science Foundation of China (Nos. 12271527 and 12471022), and Yuejian Peng was supported by 
 the National Natural Science Foundation of Hunan Province (No. 2025JJ30003) and 
the National Natural Science Foundation of China (Nos. 11931002 and 12371327).

\end{document}